\documentclass[11pt,reqno,a4paper]{amsart}
\usepackage{color}
\usepackage{amssymb,latexsym}
\usepackage{amsmath,amsthm}
\usepackage{hyperref}
\usepackage{titletoc}
\usepackage[latin1]{inputenc}
\usepackage[active]{srcltx}
\usepackage{graphicx}
\usepackage{setspace}
\usepackage{exscale,relsize}
\usepackage{textgreek}
\usepackage{epsfig,graphics}
\usepackage{psfrag}
\usepackage{caption}
\usepackage{subcaption}
\def\d{\mathrm{d}}
\def\N{\mathbb{N}}
\def\R{\mathbb{R}}
\def\m1{{I\!\!M}}

\renewcommand{\to}{\rightarrow}
\newcommand{\pa}{\partial}

\newcommand{\ino}{\int_{\Omega}}

\newcommand{\ainf}{\mbox{as\;}\;n\to+\infty}

\newcommand{\fo}{\forall}

\newcommand{\rife}[1]{(\ref{#1})}
\newcommand{\ov}[1]{\overline{#1}}
\newcommand{\un}[1]{\underline{#1}}

\newcommand{\scp}{\scriptstyle}
\newcommand{\sscp}{\scriptscriptstyle}
\newcommand{\dsp}{\displaystyle}

\renewcommand{\dfrac}{\displaystyle\frac}
\newcommand{\finedim}{\hspace{\fill}$\square$}
\newcommand{\intbar}{\mathop{\int\makebox(-15.5,0){\rule[6pt]{.7em}{0.3pt}}%
\kern-6pt}\nolimits}

\newcommand{\ii}{\infty}

\newcommand{\eps}{\varepsilon}
\newcommand{\dt}{\delta}

\newcommand{\sg}{\sigma}

\newcommand{\om}{\Omega}
\newcommand{\lm}{\lambda}

\newcommand{\omb}{\ov{\om}}

\newcommand{\e}[1]{{\,\dsp e^{\dsp #1}}}

\newcommand{\rl}{\mbox{\Large \textrho}_{\!\sscp \lm}}

\newcommand{\rh}{\mbox{\Large \textrho}}

\newcommand{\pl}{\psi_{\sscp \lm}}

\newcommand{\ple}{\psi_{\sscp \le}}

\renewcommand{\le}{\lm(E)}

\newcommand{\ue}{u_{\scp \eps}}
\newcommand{\we}{w_{\sscp (\eps)}}

\newcommand{\lep}{\lm_{\scp \eps}}
\newcommand{\prl}{{\textbf{(}\mathbf P_{\mathbf \lm}\textbf{)}}}
\newtheorem{theorem}{Theorem}[section]
\newtheorem{proposition}[theorem]{Proposition}
\newtheorem{lemma}[theorem]{Lemma}
\newtheorem{corollary}[theorem]{Corollary}
\newtheorem{remark}[theorem]{Remark}
\newtheorem{definition}[theorem]{Definition}
\newcommand{\brm}{\begin{remark}\rm}
\newcommand{\erm}{\end{remark}}
\newcommand{\bdf}{\begin{definition}\rm}
\newcommand{\edf}{\end{definition}}
\newcommand{\bte}{\begin{theorem}}
\newcommand{\ete}{\end{theorem}}
\newcommand{\bpr}{\begin{proposition}}
\newcommand{\epr}{\end{proposition}}
\newcommand{\ble}{\begin{lemma}}
\newcommand{\ele}{\end{lemma}}
\newcommand{\bco}{\begin{corollary}}
\newcommand{\eco}{\end{corollary}}
\newcommand{\beq}{\begin{equation}}
\newcommand{\eeq}{\end{equation}}
\newcommand{\bdm}{\begin{displaymath}}
\newcommand{\edm}{\end{displaymath}}

\newcommand{\graf}[1]{\left\{\begin{array}{ll}#1\end{array}\right.}

\def\sideremark#1{\ifvmode\leavevmode\fi\vadjust{\vbox to0pt{\vss
 \hbox to 0pt{\hskip\hsize\hskip1em \vbox{\hsize2.1cm\tiny\raggedright\pretolerance10000 \noindent #1\hfill}\hss}\vbox to15pt{\vfil}\vss}}}

\begin{document}
\numberwithin{equation}{section}
\parindent=0pt
\hfuzz=2pt
\frenchspacing

\title[Mean field equation and Onsager theory of 2D turbulence]{Non degeneracy, Mean Field Equations and the Onsager theory of 2D turbulence.}

\author{Daniele Bartolucci$^{(\dag)}$}
 \address{Daniele Bartolucci, Department of Mathematics, University of Rome {\it "Tor Vergata"},  Via della ricerca scientifica n.1, 00133 Roma,
Italy.}
\email{bartoluc@mat.uniroma2.it}

\author{Aleks Jevnikar}
 \address{Aleks Jevnikar, Department of Mathematics, University of Rome {\it "Tor Vergata"},  Via della ricerca scientifica n.1, 00133 Roma,
Italy.}
\email{jevnikar@mat.uniroma2.it}

\author{Youngae Lee}
\address{Youngae ~Lee,~National Institute for Mathematical Sciences, 70 Yuseong-daero 1689 beon-gil, Yuseong-gu, Daejeon, 34047, Republic of Korea}
\email{youngaelee0531@gmail.com}

\author{ Wen Yang}
\address{ Wen ~Yang,~Department of Applied Mathematics, Hong Kong Polytechnic University, Hung Hom, Kowloon, Hong Kong}
\email{math.yangwen@gmail.com}


%\thanks{$^{(1)}$Daniele Bartolucci, Department of Mathematics, University of Rome {\it "Tor Vergata"}, \\  Via della ricerca scientifica n.1, 00133 Roma, Italy. e-mail:bartoluc@mat.uniroma2.it}

\thanks{$^{(\dag)}$Research partially supported by FIRB project "{\em
Analysis and Beyond}",  by PRIN project 2012, ERC PE1\_11,
"{\em Variational and perturbative aspects in nonlinear differential problems}", and by the Consolidate the Foundations
project 2015 (sponsored by Univ. of Rome "Tor Vergata"),  ERC PE1\_11,
"{\em Nonlinear Differential Problems and their Applications}"}

\begin{abstract}
The understanding of some large energy, negative specific heat states in the Onsager description of 2D turbulence, seems to require
the analysis of a subtle open problem about bubbling solutions of the mean field equation. Motivated by this application we prove that, 
under suitable non degeneracy assumptions on the associated $m$-vortex Hamiltonian, the $m$-point bubbling solutions of the mean
field equation are non degenerate as well. Then we deduce that the Onsager mean field equilibrium entropy is smooth and strictly convex in the 
high energy regime on domains of second kind.
\end{abstract}
\maketitle

%\tableofcontents
\setcounter{section}{0}

\section{Introduction}
\setcounter{equation}{0}
The understanding of some large energy, negative specific heat states in the Onsager description of 2D turbulence, seems to require
the analysis of a subtle open problem about bubbling solutions of the mean field equation. 
Motivated by this application we prove that, under suitable non degeneracy assumptions on the associated $m$-vortex Hamiltonian, 
the $m$-point bubbling solutions of the mean field equation are non degenerate as well. Then we deduce that the Onsager mean field 
equilibrium entropy is smooth and strictly convex in the high energy regime on domains of second kind.

\subsection{Motivation.}
In a celebrated paper \cite{On} L. Onsager derived a statistical mechanics description of large vortex structures observed in two 
dimensional turbulence. He analysed the microcanonical ensemble relative to the Hamiltonian $N-$vortex model and 
concluded (among other things) that the maximum entropy (thermodynamic equilibrium) states could be realized by 
highly concentrated configurations, 
{where} like vortices attract each other. These highly concentrated states are observed in a negative temperature regime, 
meaning that, above a certain energy $E_0\in (0,+\ii)$ (see \rife{e0} below), the equilibrium entropy $S(E)$ is decreasing as a function of the energy $E$.
Those physical arguments has been later turned into rigorous proofs \cite{clmp1}, \cite{clmp2}, \cite{K}, \cite{KL}. We refer to \cite{ESp}, \cite{ESr} 
and references therein for a complete discussion about the Onsager theory and a more detailed account of the impact of those ideas. In particular, 
the rigorous analysis of the mean field canonical and microcanonical models was worked out in \cite{clmp2}, where it was shown that the mean field
thermodynamics in the energy range $E< E_{8\pi}$ is well described in terms of the unique \cite{BLin3},\cite{CCL},\cite{suz}, solutions with $\lm<8\pi$ 
of the mean field equation,
\begin{equation*}
\begin{cases}
-\Delta\psi=\dfrac{\dsp e^{\lambda\psi}}{\dsp \int_{\Omega}e^{\lambda\psi}}~&\mathrm{in}~\Omega,\\
\psi=0~&\mathrm{on}~\partial\Omega,
\end{cases}
\eqno \prl
\end{equation*}
where $\om$ is any smooth, bounded and simply connected domain and $E_{8\pi}\in (E_0,+\ii]$ is a critical energy whose value depends in a crucial way by $\om$, see Definition \ref{def1.3} below.  Here $\psi$ is the stream function of the flow,
$$\rl(\psi):=\dfrac{\dsp \e{\lm \psi}}{\dsp \ino \e{\lm \psi}},$$
is the vorticity density and $-\lm=\beta=\frac{1}{\kappa T_{\rm stat}}$ is minus the inverse statistical temperature, $\kappa$ being the Boltzmann constant. This result has been more recently generalized to cover the case of any smooth, bounded and connected domain in \cite{BLin3}.
\bdf
\label{def1}
{\it Let $\om\subset\R^2$ be a smooth and bounded domain. We say that $\om$ is of {\bf second kind} if {\rm $\prl$} admits a solution for $\lm=8\pi$. Otherwise $\om$ is said to be of {\bf first kind}.}
\edf

It follows from the results in \cite{CCL} and  \cite{BLin3} that $\om$ is of second kind if and only if $E_{8\pi}<+\ii$, see {Definition \ref{def1.3}} for the definition of $E_{8\pi}.$  As first proved in \cite{clmp2}, for $E\geq E_{8\pi}$ the Onsager intuitions lead to another amazing result, which is the non concavity of the equilibrium entropy. Although the statistical temperature in this model has nothing to do with the physical temperature, this is still a very interesting phenomenon since, if the entropy $S(E)$ is convex in a certain interval, then the system surprisingly "cools down" when the energy increases in that range. In other words the (statistical) specific heat is negative. With the exception of the results in \cite{clmp2} and more recently in \cite{B2}, we do not know of any progress in the rigorous analysis of this problem. Actually, it has been shown in \cite{clmp2} that $S(E)$ is not concave in $(E_{8\pi},+\ii)$ while, under a suitable set of assumptions, it has been shown in \cite{B2} that it is {strictly convex} in $(E_{*},\ii)$ for some $E_*>E_{8\pi}$ on any strictly star-shaped domain of second kind. One of our aims is to remove the assumptions in \cite{B2} and prove that in fact $S(E)$ is \un{convex} on any convex domain of second kind for any $E>E_{8\pi}$ large enough.

\medskip

The definition of domains of first/second kind was first introduced in \cite{clmp1} with an equivalent but different formulation.
We refer to \cite{CCL} and \cite{BLin3} for a complete discussion about the characterization of domains of first/second kind and related examples.

\brm {\it It is well known that any disk,  say $B_R=B_R(0)$, is of first kind and that in this case {\mbox{\rm $\prl$}} admits a solution if and only if $\lm<8\pi$. 
Any regular polygon is of first kind \cite{CCL}. However, there exist domains of first kind with non trivial topology where {\mbox{\rm $\prl$}}
admits solutions also for $\lm>8\pi$. For example $\om=B_R\setminus B_r(x_0)$, with $x_0\neq 0$ is of first kind if $r$ is small enough \cite{BLin3}. 
In this case it is well known that for any $N\geq 2$ there are solutions concentrating {\rm (}see \eqref{conc}{\rm )} at $N$ distinct points as $\lm\to 8\pi N$ \cite{EGP, KMdP}, as well as other solutions for any $\lm\notin 8\pi\N$ \cite{BdM2}, \cite{CLin2}.
\medskip

\indent It has been proved in \cite{BdM2} that there exists a universal constant $I_c>4\pi$ such that any convex domain whose isoperimetric ratio $I(\om)$ satisfies $I(\om)>I_c$ is of second kind. Also, if $Q_{a,b}$ is a rectangle whose side are $1\leq a\leq b<+\ii$ then there {exists $\eta_c\in(0,1)$} such that $Q_{a,b}$ is of second kind if and only if $\frac{a}{b}< \eta_c$, see \cite{CCL}.}
\erm

\bigskip

Let us define,
$$\mathcal{P}_{\sscp \om}=\left\{\rh\in L^{1}(\om)\,|\,\rh\geq 0\;\mbox{a.e. in}\;\om \right\},$$
and the entropy, energy and total vorticity of the density $\rh\in \mathcal{P}_{\sscp \om}$ as,
$$\mathfrak{S}(\rh)=-\ino \rh\log(\rh),\quad\mathcal{E}(\rh)=\frac12 \ino \rh G[\rh],\quad \mathfrak{M}(\rh)=\ino\rh,$$
respectively, where,
$$G[\rh](x)=\ino G(x,y)\rh(y)\,dy,\;x\in\om,$$
and $G=G_\om$ is the Green's function of $-\Delta$ with Dirichlet boundary conditions.

\bigskip

It has been shown in \cite{CCL} and  \cite{BLin3} that, if for $\lm=8\pi$ a solution of $\prl$ exists, say $\psi_{\sscp 8\pi}$, then it is unique. In particular it has been shown there that $\psi_{\sscp 8\pi}$ exists if and only if it is the uniform limit of the unique solutions of $\prl$ with $\lm<8\pi$. In particular one can see that $\mathcal{E}(\rl(\pl))\to +\ii$ as $\lm\to 8\pi^-$ if and only if $\om$ is of first kind. Consequently the following definition is well posed,

\bdf
\label{def1.3}{\it If $\om$ is of {second kind}, then we define $E_{8\pi}=\mathcal{E}(\mbox{\rm $\rh$}_{8\pi}(\psi_{\sscp 8\pi}))$, while if $\om$ is of {first kind} we set $E_{8\pi}=+\ii$.}
\edf

For any  $E\in(0,+\infty)$, let us consider the Microcanonical Variational Principle,
$$S(E)=\sup \left\{ \mathfrak{S}(\rh),\; \rh \in \mathcal{M}_E \right\},\quad
\mathcal{M}_E=\left\{\rh\in\mathcal{P}_{\sscp \om}\,|\,\mathcal{E}(\rh)=E,\;\;\mathfrak{M}(\rh)=1\right\}.\eqno \mbox{\bf (MVP)}$$

Among many other things which we will not discuss here, the following facts has been proved in \cite{clmp2}:
\medskip

\textbf{(MVP1)}: for each $E> 0$ there exist $\le\in\R$ and a solution $\pl=\ple$ of $\prl$ such that $\rh_{\sscp \le}(\ple)$ solves the ${\bf (MVP)}$, $S(E)=\mathfrak{S}\left(\rh_{\sscp \le}(\ple)\right)$. In particular, if $\rh$ solves the ${\bf (MVP)}$ for a certain energy $E$, then $\psi=G[\rh]$ solves $\prl$ for some $\lm=\le$, where $\pl=\psi$ and $\rh=\rl(\pl)$.
\medskip

\textbf{(MVP2)}: $S(E)$ is continuous and
\beq\label{conc}
\rh_{\sscp \lm(E)}(\ple)\rightharpoonup \dt_{x_0},\quad \mbox{as}\quad E\to +\infty,
\eeq
where $x_0$ is a maximum point of the Robin function $\gamma_{\sscp \om}(x):=R(x,x)$, and $R(x,y)=G(x,y)+\frac{1}{2\pi}\log(|x-y|)$.
\medskip

{For the sake of completeness let us recall that, by the Jensen inequality, we have,  $\mathfrak{S}(\rl(\psi_\lambda))\leq\log(|\om|),$  where the equality holds if and only if  $\rl(\psi_\lambda)(x)\equiv \mbox{\Large \textrho}_{0}(\psi_0)(x)=\frac{1}{|\om|},\,\fo x\in\ov{\om}$.  Therefore $\rh_0$ is the unique solution of the ${\bf (MVP)}$ at energy $E=E_0$, that is,}
%For the sake of completeness let us recall that, by the Jensen inequality, we have, $\mathfrak{S}(\rl)\leq\log(|\om|),$ where the equality holds if and only if $\rl(x)\equiv \rh_0(x)=\frac{1}{|\om|},\,\fo x\in\ov{\om}$. Therefore $\rh_0$ is the unique solution of the ${\bf (MVP)}$ at energy $E=E_0$, that is,
\beq
\label{e0}
E_0:=\mathcal{E}\left(\rl(\pl)\right)\left.\right|_{\lm=0}=\mathcal{E}\left(\rh_0\right)=\frac{1}{{2}|\om|^2}\ino \ino G(x,y) dx dy.
\eeq
Moreover, it has been proved in \cite{clmp2} that if $\om$ is simply connected, then for each $E<E_{8\pi}$ there exists $\lm=\lm(E)\in (-\infty,8\pi)$ 
such that $\lm(E)$ is well defined, continuous and monotonic increasing, $\lm(E)\nearrow 8\pi$ as 
$E \nearrow  E_{8\pi}$ and $S(E)=\mathfrak{S}\left(\rh_{\sscp \le}(\ple)\right)$ is smooth, concave and $\frac{dS(E)}{dE}=-\lm(E)$ in $(0,E_{8\pi})$.  
This result has been recently generalized to cover the case where $\om$ is any bounded and connected domain in \cite{BLin3}. 
In particular, see \cite{clmp2}, $S(E)$ is strictly decreasing for $E>E_0$ 
and if  $\om$ is a simply connected domain of second kind, 
then there exists constants $C_1<C_2$ such that $-8\pi E +C_1\leq S(E)\leq -8\pi E +C_2$, for any $E>E_{8\pi}$ and $S(E)$ \underline{is not concave} 
for $E>E_{8\pi}$.

\medskip

Our first main result is the following:
\bte
\label{thm1}
Let $\om$ be a convex domain of second kind.  There exists $E_{*}> E_{8\pi}$ such that $S(E)$ is smooth and strictly convex in $[E_{*},+\ii)$. In particular,
$$ \beta(E):=\frac{dS(E)}{d E}=-\lm(E),\;E\in[E_{*},+\ii), $$
where $\lm(E)$ is smooth, strictly decreasing and $\lm(E)\searrow 8\pi^+$ as $E\to +\ii$.
\ete

\brm{\it Actually the proof shows that a slightly stronger result holds, that is, the claim of Theorem \ref{thm1} is true on any strictly starshaped domain of second kind such that $\gamma_{\sscp \om}$ has a unique and non degenerate maximum point.}
\erm

Even with the aid of the non degeneracy and uniqueness results derived here (see Theorem~\ref{th1.1} below) and in \cite{BJLY} 
{(see Theorem 3.B in section \ref{sec5} below),} the proof of Theorem \ref{thm1} is not straightforward, because the monotonicity of 
the Lagrange multiplier $\lm(E)$ is not easy to check. We succeed in finding a solution to this problem by showing that 
$\lm(E)=\left.\lm_\eps\right|_{\eps(E)}$ where $\lm_\eps$ is a monotone function constructed by using the so called Gelfand equation (see \rife{peps}) 
and a related non degeneracy property \cite{GlGr}, see {Theorem 3.C} in section \ref{sec5}. The proof of this part also relies on another delicate result about the characterization of domains of second kind \cite{BLin3}. This workaround eventually solves the problem since it turns out that the energy as a function of $\eps$ has the "right" monotonicity property, see Lemma \ref{enrgeps}.

\medskip

\subsection{The mean field equation.}
The proof of Theorem \ref{thm1} is based on various results, see section \ref{sec5}, of independent interest about the mean field equation,
\beq\label{1lm}
\begin{cases}
-\Delta u_{\lambda}=\lambda\dfrac{\dsp he^{u_{\lambda}}}{\dsp \int_{\Omega}he^{u_{\lambda}}}~&\mathrm{in}~\Omega,\\
u_{\lambda}=0~&\mathrm{on}~\partial\Omega,
\end{cases}
\eeq
and about the Gelfand equation \rife{peps}, where $\om$ is a smooth and bounded domain, $h(x)=\hat h(x)e^{-4\pi\sum_{i=1}^\ell \alpha_iG(x,p_i)}\geq0$, the $p_i$'s are distinct points in $\om$, $\alpha_i>-1$ for any $i=1,\cdots,\ell$, $\hat h$ is a positive smooth function in $\ov{\om}$ and $G(x,p)$ is the Green function satisfying,
\begin{equation*}
-\Delta G(x,p)=\delta_p~\mathrm{in}~\Omega,\quad G(x,p)=0~\mathrm{on}~\partial\Omega.
\end{equation*}

Because of the many motivations in pure and applied mathematics \cite{Aub,Ban,bav,beb,clmp1,Gel,On,suzC,T,Troy,yang}, 
a lot of work has been done to understand these pde's and we just mention few results \cite{BdM2,BdMM,bt2,CCL,CLin5,EGP,Fang,GM1,yy,linwang,Mal2,suz,tar} 
and the references quoted therein. Among many other things, the interesting properties of these problems are the lack of compactness \cite{bt,bm,yy}, 
which in turn causes either non existence or non uniqueness of solutions of \rife{1lm} with $\lm>8\pi$ \cite{CLin2,dem2, Mal1}, see also \cite{B2,clmp2}.
As a consequence, in general, the global bifurcation diagram could be very complicated and in particular solutions may be degenerate. 
Here we say that a solution of \rife{1lm} is degenerate if the corresponding linearized problem (see \rife{1.8l} below) admits a non trivial solution.
This is why it is important to understand which are the regions in the bifurcation diagram where one can recover uniqueness and {nondegeneracy.}

\bdf
\label{def1.6}
{\it A sequence of solutions $u_n:=u_{\sscp \lm_n}$ of \eqref{1lm} is said to be an $m$-bubbling sequence if,
$$\lm_n\dfrac{\dsp h\e{u_n}}{ \dsp  \ino h \e{u_n}}\rightharpoonup 8\pi \sum\limits_{j=1}^m \dt_{q_j},\;\mbox{as}\;n\to +\ii,$$
weakly in the sense of measures in $\om$, where $\{q_1,\cdots, q_m\}\subset \om$ are $m$ distinct points satisfying $\{q_1,\cdots, q_m\}\cap \{p_1,\cdots,p_\ell\}=\emptyset$. The points $q_j$'s are said to be the blow up points and $\{q_1,\cdots, q_m\}$ the blow up set.}
\edf

\brm
\label{rem1.7} {\it It is well known \cite{bm,CCL,yy}, that if $u_n$ is an $m$-bubbling sequence, then, as far as the domain is not too irregular, see Remark 
\ref{rem1.9}, then 
necessarily $\lm_n\to 8\pi m$, as $n\to +\ii$.}
\erm

The uniqueness for $\lm\leq 8\pi$ is a delicate problem which has been discussed in \cite{bl,BLin3, BJLY, CCL, suz}. On the other side, we do not know of any general {nondegeneracy} result for bubbling solutions of \rife{1lm}. Actually  
the nondegeneracy of $m$-bubbling solutions of the Gelfand problem (see \rife{peps} below) was first established 
in \cite{GlGr} for $m=1$ and then in \cite{GOS} for $m\geq 1$. It seems however that our problem is rather different, since our linearized 
equation contains an additional term, see \rife{1.8l}. Therefore, after a suitable translation, we would end up with a solution of an equation 
with vanishing weighted mean value and constant and unknown value at the boundary, unlike the Gelfand problem where the solution satisfies the 
natural Dirichlet boundary conditions. Let us consider a generic $m$-bubbling sequence which satisfies,
\begin{equation}
\label{1.1l}
\graf{\Delta u_n+\lambda_n\dfrac{\dsp h\e{u_n}}{\dsp \int_{\Omega} h\e{u_n}} =0~&\mathrm{in}~\Omega,\\
u_n=0~&\mathrm{on}~\partial\Omega,}
\end{equation}
for any $n\in\N$, and prove that $u_n$ is non-degenerate, provided that $n$ is sufficiently large, and $h$ and $\om$ satisfy some
other non-degeneracy assumptions. Let
$$R(x,y)=G(x,y)+\frac{1}{2\pi}\log|x-y|,$$
be the regular part of the Green function $G(x,y)$. For $\mathbf{q}=(q_1,\cdots,q_m)\in\om\times\cdots\times\om,$
we set
\begin{equation}
\label{1.2}
G_j^*(x)=8\pi R(x,q_j)+8\pi\sum_{l\neq j}^{1,\dots,m}G(x,q_l),
\end{equation}
\begin{equation}
\label{1.3}
\ell(\mathbf{q})=\sum_{j=1}^m[\Delta \log h(q_j) ]h(q_j)e^{G_j^*(q_j)},
\end{equation}
and
\begin{equation}
\label{1.4}
f_{\mathbf{q},j}(x)=8\pi\Big[R(x,q_j)- R(q_j,q_j)+\sum^{1,\dots,m}_{l\neq j}(G(x,q_l)-G(q_j,q_l))\Big]+\log \frac{h(x)}{ h(q_j)}.
\end{equation}

We will denote by $B_r(q)$ the ball of radius $r$ centred at $q\in\Omega.$ For the case $m\geq2$ we fix a constant $r_0\in(0,\frac12)$ and a family of open sets
$\Omega_j$ satisfying $\Omega_l\cap\Omega_j=\emptyset$ if $l\neq j$, $\bigcup_{j=1}^m\overline \om_j=\overline\om$,
$B_{2r_0}(q_j)\subset\Omega_j,~j=1,\cdots,m$. Then, let us define,
\begin{equation}
\label{1.5}
D(\mathbf{q})=\lim_{r\to0}\sum_{j=1}^m h(q_j)e^{G_j^*(q_j)}\Big(\dsp\int_{\Omega_j\setminus B_{r_j}(q_j)}
e^{\Phi_j(x,\mathbf{q})} \mathrm{d}x-\frac{\pi}{r_j^2}\Big),
\end{equation}
where $\Omega_1=\Omega$ if $m=1$, $r_j=r\sqrt{8h(q_j)e^{G_j^*}(q_j)}$ and
\begin{equation}
\label{1.6}
\Phi_j(x,\mathbf{q})=\sum_{l=1}^m 8\pi  G(x,q_l)-G_j^*(q_j)+\log h(x)-\log h(q_j).
\end{equation}
The quantity $D(\mathbf{q})$ was introduced first in \cite{CCL, clw}. For $(x_1,\cdots,x_m)\in\om\times\cdots\times\om$ we also define,
\begin{equation}
\label{1.7}
f_m(x_1,x_2,\cdots,x_m)=\sum_{j=1}^{m}\Big[\log(h(x_j))+4\pi R(x_j,x_j)\Big]+4\pi\sum_{l\neq j}^{1,\cdots,m}G(x_l,x_j),
\end{equation}
and let $D_\Omega^2f_m$ be its Hessian matrix on $\Omega$. The function $f_m$ is also known in literature as the $m$-vortex Hamiltonian \cite{New}. We consider the linearized problem relative to \eqref{1.1l},
\begin{equation}
\label{1.8l}
\graf{\Delta\phi+\lambda_n\dfrac{\dsp he^{u_n}}{\dsp \int_{\Omega}he^{u_n}\d y}\left(\phi-\dfrac{\dsp\int_{\Omega}he^{u_n}\phi\,\d y}{\dsp\int_{\Omega}he^{u_n}\d y}\right)=0~&\mathrm{in}~\Omega,\\
\phi=0~&\mathrm{on}~\partial\Omega.}
\end{equation}

Our second main result is the following,

\begin{theorem}
\label{th1.1}
Let $u_n$ be an $m$-bubbling sequence of \eqref{1.1l}
with blow up set $\{q_1,\cdots,q_m\}\cap\{p_1,\cdots,p_\ell\}=\emptyset$, where
$\mathbf{q}=(q_1,\cdots,q_m)$ is a critical point of $f_m$ and let $det(D_{\Omega}^2f_m(\mathbf{q}))\neq0$. Assume that either,
\begin{enumerate}
  \item [(1)] $\ell(\mathbf{q})\neq0$, or,
  \item [(2)] $\ell(\mathbf{q})=0$ and $D(\mathbf{q})\neq0$.
\end{enumerate}
Then there exists $n_0\geq1$ such that, for any $n\geq n_0$, \eqref{1.8l} admits only the trivial solution $\phi\equiv 0$.
\end{theorem}

\brm\label{rem1.9} {\it Actually, Theorem \ref{th1.1} holds for a wider class of domains $\om$. More precisely, as in \cite{CCL} we say that $\om$ is regular if its boundary $\partial \om$ is of class $C^2$ but
for a finite number of points $\{Q_1 , . . . , Q_{N_0}\}\subset \partial \om$ such that the following conditions holds at each $Q_j$.
\begin{enumerate}
  \item [(i)] The inner angle $\theta_j$ of $\partial \om$ at $Q_j$ satisfies $0 < {\theta_j \neq \pi} < 2\pi$;
  \item [(ii)] At each $Q_j$ there is an univalent conformal map from $B_\delta (Q_j) \cap \overline{\om}$ to
the complex plane $\mathbb{C}$ such that $\partial \om \cap B_\delta (Q_j)$ is mapped to a $C^2$ curve.
\end{enumerate}
It can be shown by a moving plane argument (see \cite{CCL}) that solutions of \eqref{1lm} are uniformly bounded in a fixed neighborhood of $\partial\om$. Since the argument of Theorem \ref{th1.1} relies on the local estimates for blow up solutions of \eqref{1lm} it can be carried out for the latter class of domains as well. Hence, Theorem \ref{th1.1} holds for any regular domain $\om$. {See also Theorem 3.A below.}
\medskip

On the other hand, we may consider a more general problem than \eqref{1lm} in 
which the weight $h$ is replaced by $h_n(x)=\hat h_n(x)e^{-4\pi\sum_{i=1}^\ell \alpha_{i,n}G(x,p_{i,n})}$, 
where $p_{i,n}\to p_i$, $\alpha_{i,n}\geq-1+\frac 1C$, $\alpha_{i,n}\to\alpha_i$ for any $i=1,\dots,\ell$ and 
$\hat h_n$ are smooth functions such that $\hat h_n\geq \frac 1C$ in $\overline \om$, $\hat h_n\to\hat h_n$ in $C^{2,\sg}(\overline\om)$, 
for some fixed constant $C>0$ and $\sg>0$. It is straightforward to check that the latter generalization does not affect the argument of Theorem \ref{th1.1} and thus Theorem \ref{th1.1} holds for this class of problems as well.}
\erm

To prove Theorem \ref{th1.1} we will analyze the asymptotic behavior of the auxiliary sequence,
$$\zeta_n=\dfrac{\phi_n-\frac{\int_{\Omega}he^{u_n}\phi_n\d y}{\int_{\Omega}he^{u_n}\d y}}{\left\|\phi_n-\frac{\int_{\Omega}he^{u_n}\phi_n\d y}{\int_{\Omega}he^{u_n}\d y}\right\|_{L^{\infty}(\Omega)}},$$
where $\phi_n$ is a solution of \eqref{1.8l}. Near each blow up point $q_j$ and after a suitable scaling, $\zeta_n$ converges to an
{element in the kernel space of the linearized operator $L$, where}
%entire solution of the linearized problem associated to the Liouville equation:
%\begin{equation}
%\label{1.9}
%\Delta v+ e^{v}=0\quad\textrm{in}\quad \mathbb{R}^2,\quad \int_{\mathbb{R}^2}e^v<+\infty.
%\end{equation}
%The solutions of \eqref{1.9} are classified and take the following form, see \cite{cli1},
%\begin{equation}
%\label{1.10}
%v(z)=v_{\mu,a}(z)=\log\frac{8e^{\mu}}{(1+e^{\mu}\vert z+a\vert^{2})^{2}},\quad\mu\in\mathbb{R},\quad a=(a_1,a_2) \in \mathbb{R}^2.
%\end{equation}
%The linearized operator $L$ relative to $v_{0,0}$  is defined by,
\begin{equation}
\label{1.11}
L\phi:=\Delta \phi+\frac{8}{(1+|z|^2)^2}\phi\quad \mathrm{in}\quad \mathbb{R}^2.
\end{equation}
It is well known (\cite[Proposition 1]{bp}) that the kernel of $L$ has {real dimension $3$.} A crucial point in the proof of Theorem \ref{th1.1} is to show that, 
after suitable scaling and for large $n$, $\zeta_n$ is orthogonal to {the kernel space}. 
However this is not enough, since we also need to figure out the global behaviour of $\zeta_n$ in the zone {far away} from the blow up point. 
These problems are rather subtle and require the well known pointwise estimates \cite{CLin1}, together with some refinements recently derived in \cite{BJLY}. 
Then we will work out a delicate adaptation of some arguments in \cite{BJLY,ly2} based on the analysis of various suitably defined Pohozaev-type identities.
\medskip

This paper is organized as follows. In section \ref{sec2}, we prove Theorem \ref{th1.1}. In section \ref{sec5}, we review the known results for domains of second kind, the local uniqueness for bubbling solutions of mean field equation and the non degeneracy result for the Gelfand problem. In section \ref{sec6}, we prove Theorem \ref{thm1}. In the Appendix we provide the technical estimation for the proof of Theorem \ref{th1.1}.

\section{{Proof of Theorem \ref{th1.1}}}\label{sec2}
\subsection{{Preliminaries}}
{Suppose} that $u_n$ is a sequence of blow-up solutions of \eqref{1.1l} which blows up at $q_j\notin\{p_1,\cdots,p_\ell\},~j=1,\cdots,m.$ Let
\begin{equation*}
\tilde u_n(x)=u_n(x)-\log\left(\int_{\Omega}he^{u_n}\d y\right),
\end{equation*}
then it is easy to see that,
\begin{equation}
\label{2.1}
\Delta \tilde u_n+\lambda_nh(x)e^{\tilde u_n(x)}=0~\mathrm{in}~\Omega\quad\mathrm{and}\quad \int_{\Omega}he^{\tilde{u}_n}\d y=1.
\end{equation}
We denote by,
\begin{equation}
\label{2.2}
\mu_n=\max_\Omega \tilde{u}_n,\quad \mu_{n,j}=\max_{B_{r_0}(q_j)}\tilde u_n=\tilde u_n(x_{n,j})\quad \mathrm{for}\quad j=1,\cdots,m.
\end{equation}
Let us define,
\begin{equation}
\label{2.3}
U_{n,j}(x)=\log\frac{e^{\mu_{n,j}}}{(1+\frac{\lambda_nh({x}_{n,j})}{8}e^{\mu_{n,j}}|{x}-{x}_{n,j,*}|^2)^2},\quad  {x}\in \mathbb{R}^2,
\end{equation}
where the point $x_{n,j,*}$ is chosen to satisfy,
$$\nabla U_{n,j}({x}_{n,j})=\nabla(\log h({x}_{n,j})).$$
It is not difficult to check that,
\begin{align}
\label{2.4}
|x_{n,j}-x_{n,j,*}|=O(e^{-\mu_{n,j}}).
\end{align}
{Near the blow up point $q_j$, we let $\eta_{n,j}$ denote the "error term",}
\begin{equation}
\label{2.5}
\eta_{n,j}(x)=\tilde{u}_n(x)-U_{n,j}(x)-(G_{j}^{*}(x)-G_{j}^{*}(x_{n,j})), \quad {x}\in B_{r_0}(x_{n,j}).
\end{equation}
{For $x\in \omb\setminus \cup_{j=1}^m B_{\tau}(q_j))$,  $\tilde{u}_n$ should be well approximated by a sum of Green's function, which we denote by,}
\begin{equation}
\label{2.15}
w_n(x)=\tilde{u}_n(x)-\sum_{j=1}^m\rho_{n,j} G(x,x_{n,j})-\tilde{u}_{n,0},
\end{equation}{where   here and in the rest of this work, we set,}
\begin{equation}\label{boundary_u}
{\tilde{u}_{n,0}=\tilde u_n|_{\partial\Omega}.}
\end{equation}
{At each blow up point $q_j,~1\leq j\leq m$, we also define the local mass by}
\begin{equation}
\label{2.12}
\rho_{n,j}=\lambda_n \int_{B_{r_0}(q_j)}he^{\tilde{u}_n}\d y.
\end{equation}

The following sharp estimates for blow up solutions of \eqref{1.1l} have been obtained in \cite{CLin1}. The higher order term in (vi) has been obtained more 
recently in \cite{BJLY}.
\medskip

{\bf { Theorem 2.A. (\cite{CLin1})}} \textit{{ Suppose  that $u_n$ is a sequence of blow-up solutions of \eqref{1.1l} which blows up at 
$q_j\notin\{p_1,\cdots,p_\ell\},~j=1,\cdots,m.$ Then the following facts hold: }}

 { (i) $\eta_{n,j}=O(\mu_{n,j}^2 e^{-\mu_{n,j}})$ \textit{on} $B_{r_0}(x_{n,j})$;}

 { (ii) $w_n=o(e^{-\frac{\mu_{n,j}}{2}})$ \textit{on} $C^1(\omb\setminus \cup_{j=1}^m B_{\tau}(q_j))$; }

 { (ii) $e^{\mu_{n,j}}h^2(x_{n,j})e^{G_j^*(x_{n,j})}=e^{\mu_{n,1}}h^2(x_{n,1})e^{G_1^*(x_{n,1})}(1+O(e^{-\frac{\mu_{n,1}}{2}}))$; }

 { (iii) $\mu_{n,j}+\tilde{u}_{n,0}+2\log\left(\frac{\lambda_n h(x_{n,j})}{8}\right)+G_j^*(x_{n,j})
=O(\mu_{n,j}^2 e^{-\mu_{n,j}})$;}

 { (iv) $\nabla[\log h(x)+G_j^*(x)]\Big|_{x=x_{n,j}}=O(\mu_{n,j} e^{-\mu_{n,j}})$ \textit{and} $|{x}_{n,j}-{{q}}_j|=O(\mu_{n,j} e^{-\mu_{n,j}})$;}

 { (v) $\rho_{n,j}-8\pi= O(\mu_{n,j}e^{-\mu_{n,j}})$; }

 { (vi) \textit{for a fixed small constant} $\tau>0$,}
\begin{equation*}
\begin{aligned}
\label{2.14}
\lambda_n-8\pi m=~&\frac{2\ell(\mathbf{q})e^{-\mu_{n,1}}}{m h^2(q_1)e^{G_1^*(q_1)}}\Big(\mu_{n,1}+\log\left(\lambda_n h^2(q_1)e^{G_1^*(q_1)}\tau^2\right)
-2\Big)\\
&+\frac{8e^{-\mu_{n,1}}}{h^2(q_1)e^{G_1^*(q_1)}\pi m}\Big(D(\mathbf{q})
+O(\tau^\sigma)\Big)+O(\mu_{n,1}^2e^{-\frac{3}{2}\mu_{n,1}})+O(e^{-(1+\frac{\sigma}{2})\mu_{n,1}}),
\end{aligned}
\end{equation*}
\textit{where} $\sigma>0$ \textit{is a positive number that satisfies} $\hat h\in C^{2,\sigma}(\ov{\om})$.

\subsection{{The proof of Theorem \ref{th1.1}}} We shall prove Theorem \ref{th1.1} by contradiction. Suppose that \eqref{1.8l} admits a nontrivial solution $\phi_n,$ where $u_n$ is an $m$-bubbling sequence of solutions of \eqref{1.1l} with blow up points $q_j\notin\{p_1,\cdots,p_\ell\}$, $j=1,\cdots,m$, such that $\mathbf{q}=(q_1,\cdots,q_m)$ is a critical point of $f_m$ and $\textrm{det}(D_{\Omega}^2f_m(\mathbf{q}))\neq 0$. Next, let us define,
\begin{equation}
\label{3.1}
\zeta_n(x)=\frac{\phi_n(x)-\frac{\int_{\Omega}he^{u_n}\phi_n\mathrm{d}y}{\int_{\Omega}h e^{u_n}\mathrm{d}y}}
{\left\|\phi_n-\frac{\int_{\Omega}he^{u_n}\phi_n\mathrm{d}y}{\int_{\Omega}he^{u_n}\mathrm{d}y}\right\|_{L^\infty(\Omega)}}
\quad\mathrm{and}\quad
f_n^*(x)=\lambda_nh(x)e^{\tilde{u}_n(x)}\zeta_n(x).
\end{equation}
It is easy to see that,
\begin{equation}
\label{3.2}
\int_{\Omega}f_n^*dx=0,
\end{equation}
and $\zeta_n$ satisfies,
\begin{equation}
\label{3.3}
\begin{aligned}
\Delta \zeta_n+f_n^*(x)=\Delta \zeta_n+\lambda_nh(x)c_n(x)\zeta_n(x)=0,
\end{aligned}
\end{equation}
where\begin{equation}
\label{cn}
\begin{aligned}{c_n(x)=e^{\tilde{u}_n}(x).}\end{aligned}
\end{equation}
{In order to prove Theorem \ref{th1.1}, we shall follow the argument in \cite{ly2, BJLY}. In \cite{ly2,BJLY}, 
the authors used various Pohozaev identities to show the local uniqueness of bubbling solutions. 
In our case, $\zeta_n$ plays the role of the difference of two solutions in \cite{ly2,BJLY} (see Step 3 below).   }

\medskip
\noindent
{Step 1. The asymptotic behavior of $\zeta_n$ near blow up points.} \\
{In this step we shall prove that there exist constants $b_{j,0}$, $b_{j,1}$, and $b_{j,2}$ such that,}
\begin{equation}\label{nearblow}
{ \zeta_n\left(e^{-\frac{\mu_{n,j}}{2}}z+{x}_{n,j}\right)\to  b_{j,0}\psi_{j,0}(z)+b_{j,1}\psi_{j,1}(z)+b_{j,2}\psi_{j,2}(z)
~\mbox{in}~C^0_{\mathrm{loc}}(\mathbb{R}^2),}
\end{equation}
{where,}
\begin{equation*}
{\psi_{j,0}(z)=\frac{1-\pi m h(q_j)|z|^2}{1+\pi m h(q_j)|z|^2},~\psi_{j,1}(z)=\frac{\sqrt{\pi mh(q_j)}z_1}{1+\pi m h(q_j)|z|^2},~
\psi_{j,2}(z)=\frac{\sqrt{\pi mh(q_j)}z_2}{1+\pi m h(q_j)|z|^2}.}
\end{equation*}
{Indeed,  it is well that, after a suitable scaling, $\zeta_{n}$ converges to a function 
$\zeta_j(z)$ in $C^0_{\mathrm{loc}}(\mathbb{R}^2)$,
where $\zeta_j(z)$ satisfies,}
\begin{equation*}{
\Delta \zeta_j+\frac{8\pi m h({q}_j)}{(1+\pi m h({q}_j)|z|^2)^2}\zeta_j(z)=0~\mathrm{in}~\mathbb{R}^2,\quad |\zeta_j|\leq 1.}
\end{equation*}
Then  \eqref{nearblow} follows by \cite[Proposition 1]{bp}.
\medskip

\noindent
{Step 2. The global behavior of $\zeta_n$ far away from the blow up points.}\\
{We claim that there exists  a constant $b_0$ such that,}
\begin{equation}\label{farblow}
{\zeta_n\to-b_0\ \ \textrm{in}\ \ C^0_{\textrm{loc}}( \ov{\om}\setminus\bigcup_{j=1}^m \{q_j\}), \ \ \textrm{and} \  b_0=b_{j,0} \ \ \textrm{for}\  j=1,\cdots, m.}
\end{equation}
{By Theorem 2.A and  $\|\zeta_n\|_{L^\infty(\Omega)}\leq 1$,
we see that $\zeta_n\to\zeta_0$ in $C^0_{\mathrm{loc}}(\omb\setminus\{q_1,\cdots,q_m\})$, where $\zeta_0$ satisfies,}
\begin{equation}
\label{3.14}
\Delta\zeta_0=0~\mathrm{in}~\Omega\setminus\{q_1\cdots q_m\}.
\end{equation}
{By $\|\zeta_n\|_{L^\infty(\Omega)}\leq 1$,  we can extend \eqref{3.14} to $\Omega$.
Then we conclude that,}
\begin{equation}
\label{limit}
\zeta_0\equiv-b_0 \ \textrm{in}\ \  \Omega, \ \textrm{ where}\  b_0\ \textrm{ is a constant}.
\end{equation}
{In order to prove $b_{j,0}=b_0$ for $j=1,\cdots, m$, let us set,}  \[\psi_{n,j}(x)=\frac{1-\frac{\lambda_n}{8}h({x}_{n,j})|x-x_{n,j}|^2e^{\mu_{n,j}}}
{1+\frac{\lambda_n}{8}h(x_{n,j})|x-x_{n,j}|^2e^{\mu_{n,j}}}\ \ \textrm{and}\ \ \zeta_{n,j}^*(r)=\int_0^{2\pi}\zeta_n(r,\theta)\mathrm{d}\theta, \ \textrm{ where}\  r=|x-x_{n,j}|.\]
{Fix large $R>0$. By \eqref{3.3} and Theorem 2.A, we can derive  the following estimate  (see \cite{ly2} and \cite[Lemma  3.4]{BJLY} for details) }
\begin{equation}\label{int_est}
\begin{aligned}&{\frac{1}{r}\int_{\partial B_d(x_{n,j})}\left(\psi_{n,j}\frac{\partial\zeta_n}{\partial\nu}
-\zeta_n\frac{\partial\psi_{n,j}}{\partial \nu}\right)\d\sigma
=(\zeta_{n,j}^*)'(r)\psi_{n,j}(r)-\zeta_{n,j}^*(r)\psi_{n,j}'(r)=\frac{o\left(\frac{1}{\mu_{n,j}}\right)}{r},}\,\end{aligned}\end{equation} for $r\in (Re^{-\mu_{n,j}/2},r_0).$
By using \eqref{int_est}, we conclude that,
\begin{equation}
\label{3.19}
\zeta_{n,j}^*(r)=-2\pi b_{j,0}+o_R(1)+o_n(1)(1+O(R))~\mbox{for all}~r\in (Re^{-\mu_{n,j}/2},r_0),
\end{equation}where $\lim_{n\to+\infty}o_n(1)=0$ and  $\lim_{R\to+\infty}o_R(1)=0$.
This fact together with \eqref{limit}, shoes that $b_{j,0}=b_0$ for $j=1,\cdots, m$ and thus the claim \eqref{farblow} follows.
\medskip

\noindent
{Step 3. The Pohozaev type identities.}\\
{The following Pohozaev type identities play a crucial role in the proof of Theorem \ref{th1.1}. {For any} fixed $r\in(0,r_0)$, the followings hold:}
\begin{equation}
\begin{aligned}
\label{4.3} &\int_{\partial B_r(x_{n,j})}r\langle Dv_{n,j},D\zeta_n\rangle\mathrm{d}\sigma
-2\int_{\partial B_r({x}_{n,j})}r\langle\nu,D v_{n,j}\rangle\langle\nu,D\zeta_n\rangle\mathrm{d}\sigma\\
&=\int_{\partial B_r({x}_{n,j})}r\lambda_nhe^{\tilde{u}_n}\zeta_n\mathrm{d}\sigma
-\int_{B_r(x_{n,j})}\lambda_nh({x})e^{\tilde{u}_n}\zeta_n(2+\langle D(\log h+\phi_{n,j}),x-x_{n,j}\rangle)\mathrm{d}x,
\end{aligned}
\end{equation}and for $i=1,2$, \begin{equation}
\begin{aligned}
\label{4.47} &\int_{\partial B_r(x_{n,j})}\langle\nu, D\zeta_n\rangle D_iv_{n,j}+\langle\nu, Dv_{n,j}\rangle D_i\zeta_n\mathrm{d}\sigma
-\int_{\partial B_r(x_{n,j})}\langle D {v_{n,j}} ,D\zeta_n\rangle\frac{(x-x_{n,j})_i}{|x-x_{n,j}|}\mathrm{d}\sigma\\
&=-\int_{\partial B_r(x_{n,j})}\lambda_nh(x) e^{\tilde{u}_n}\zeta_n\frac{(x-x_{n,j})_i}{|x-x_{n,j}|}\mathrm{d}\sigma
+\int_{B_r(x_{n,j})}\lambda_nh(x)e^{\tilde{u}_n}\zeta_nD_i(\phi_{n,j}+\log h)\mathrm{d}{x}.
\end{aligned}
\end{equation}
where\begin{equation}
\label{4.2}
v_{n,j}(y)=\tilde{u}_n(y)-\phi_{n,j}(y)\ j=1,\cdots,m, \ \textrm{and}
\end{equation}
\begin{equation}
\label{4.1}
\phi_{n,j}(y)=\frac{\lambda_n}{m}\Big[(R(y,x_{n,j})-R(x_{n,j},x_{n,j}))+ \sum_{l\neq j}(G(y,x_{n,l})-G(x_{n,j},x_{n,l}))\Big].
\end{equation}
The proof of \eqref{4.3} is obtained by using,
\begin{equation*}
\Delta\zeta_n+\lambda_nh(x)e^{\tilde{u}_n}\zeta_n=0,\quad
\Delta v_{n,j}+\lambda_nhe^{\tilde{u}_n}=0,
\end{equation*}
together with,
\begin{equation}
\label{4.4}
\begin{aligned}
&\mathrm{div}\left(\nabla\zeta_n\left(\nabla v_{n,j}\cdot(x-x_{n,j})\right)+\nabla v_{n,j}\left(\nabla\zeta_n\cdot(x-x_{n,j})\right)-  \nabla\zeta_n\cdot\nabla v_{n,j}(x-x_{n,j})\right)\\
&=\Delta \zeta_n\left(\nabla v_{n,j} \cdot({x}-{x}_{n,j})\right)+\Delta v_{n,j}\left(\nabla\zeta_n\cdot({x}-{x}_{n,j})\right)\\
&=-\lambda_ne^{v_{n,j}+\phi_{n,j}+\log h}\zeta_n\left(\nabla v_{n,j}\cdot(x-x_{n,j})\right)
-\lambda_ne^{v_{n,j}+\phi_{n,j}+\log h}\left(\nabla\zeta_n\cdot(x-x_{n,j})\right)\\
&=-\mathrm{div}\left(\lambda_n e^{v_{n,j}+\phi_{n,j}+\log h}\zeta_n(x-x_{n,j})\right)+2\lambda_n e^{v_{n,j}+\phi_{n,j}+\log h}\zeta_n\\
&~\quad+\lambda_ne^{v_{n,j}+\phi_{n,j}+\log h}\zeta_n\left(\nabla(\phi_{n,j}+\log h)\cdot(x-x_{n,j})\right).
\end{aligned}
\end{equation}
While the proof of \eqref{4.47} follows by using the equation $\Delta v_{n,j}+\lambda_nh e^{\tilde{u}_n}=0$ and
\begin{equation*}
\begin{aligned}
&\mathrm{div}\left(\nabla\zeta_nD_lv_{n,j}+\nabla v_{n,j}D_l\zeta_n-\nabla\zeta_n\cdot\nabla v_{n,j}e_l\right)\\
&=\Delta\zeta_nD_lv_{n,j}+\Delta v_{n,j} D_l\zeta_n=-\lambda_nh(x) e^{\tilde{u}_n} \zeta_nD_lv_{n,j}-\lambda_nh(x)e^{\tilde{u}_n}D_l\zeta_n\\
&=-\lambda_ne^{v_{n,j}+\phi_{n,j}+\log h}\zeta_nD_l(v_{n,j}+\phi_{n,j}+\log h)-\lambda_ne^{v_{n,j}+\phi_{n,j}+\log h} D_l\zeta_n\\
&\quad+\lambda_ne^{v_{n,j}+\phi_{n,j}+\log h}\zeta_nD_l(\phi_{n,j}+\log h)\\
&=-\mathrm{div}(\lambda_ne^{v_{n,j}+\phi_{n,j}+\log h}\zeta_ne_l)
+\lambda_ne^{v_{n,j}+\phi_{n,j}+\log h}\zeta_nD_l(\phi_{n,j}+\log h),
\end{aligned}
\end{equation*}
where $e_l=\frac{(x-x_{n,j})_l}{|x-x_{n,j}|}$, $l=1,2$.

\medskip
\noindent
{Step 4. Based on the identities established in the last step, we can prove that 
\begin{equation}
\label{bequal0}
b_0=b_{j,0}=b_{j,1}=b_{j,2}=0,~j=1,\cdots,m.
\end{equation}
The proof of the above fact is long and technical and so we postpone it to the last section \ref{appendix}.}
Once we get that \eqref{bequal0} holds, then the proof of Theorem \ref{th1.1} is almost done. Indeed, if $x_n^*$ is a maximum point of $\zeta_n$, then we have,
\begin{equation}
\label{4.56}
|\zeta_n(x_n^*)|=1.
\end{equation}
In view of Step 1-2 and \eqref{bequal0}, we find that $\lim_{n\to+\infty}x_n^*=q_j$ for some $j$ and
\begin{equation}
\label{4.57}
\lim_{n\to+\infty}e^{\frac{\mu_{n,j}}{2}}s_n=+\infty,~\mathrm{where}~s_n=|x_n^*-x_{n,j}|.
\end{equation}
Setting $\tilde{\zeta}_n(x)=\zeta_n(s_n {x}+x_{n,j})$, then from \eqref{3.3} and Theorem 2.A we see that $\tilde{\zeta}_n$ satisfies,
\begin{equation*}
\begin{aligned}
\Delta\tilde{\zeta}_n+\frac{\lambda_nh(x_{n,j})s_n^2e^{\mu_{n,j}}(1+ O(s_n|x|)+o(1))\tilde{\zeta}_n}
{(1+\frac{\lambda_nh(x_{n,j})}{8}e^{\mu_{n,j}}|s_nx+x_{n,j}-x_{n,j,*}|^2)^2}=0.
\end{aligned}
\end{equation*}
On the other hand, by \eqref{4.56} we also have,
\begin{equation}
\label{4.58}
\Big|\tilde{\zeta}_n\Big(\frac{x_n^*-x_{n,j}}{s_n}\Big)\Big|=|\zeta_n(x_n^*)|=1.
\end{equation}
In view of \eqref{4.57} and $|\tilde{\zeta}_n|\leq 1$ we see that $\tilde{\zeta}_n\to\tilde{\zeta}_0$
on any compact subset of $\mathbb{R}^2\setminus\{0\}$, where $\tilde{\zeta}_0$ satisfies $\Delta \tilde{\zeta}_0=0$
in $\mathbb{R}^2\setminus\{0\}$. Since $|\tilde{\zeta}_0|\leq 1$, we have $\Delta \tilde{\zeta}_0=0$ in $\mathbb{R}^2$, whence
$\tilde{\zeta}_0$ is a constant. At this point, since $ \frac{|x_n^*-x_{n,j}|}{s_n}=1$ and in view of \eqref{4.58}, we find that,
$\tilde{\zeta}_0\equiv1$ or $\tilde{\zeta}_0\equiv-1$. As a consequence we conclude that,
$$|\zeta_n(x)|\ge \frac{1}{2}~\mathrm{if}~s_n\leq |{x}-{x}_{n,j}|\leq 2s_n,$$
which contradicts \eqref{3.19},  since $e^{-\frac{\mu_{n,j}}{2}}<< s_n$, $\lim_{n\to+\infty}s_n=0$, and $b_0=b_{j,0}=0$.
This fact concludes the proof of Theorem \ref{th1.1}.\hfill$\Box$

\medskip
\section{The proof of Theorem \ref{thm1}: preliminary results}\label{sec5}
First of all, we will need the following characterization of domains of second kind recently derived in \cite{BLin3}. 
Let us recall that here $\gamma_{\sscp \om}(x)=R(x,x)$.

\medskip

{\bf Theorem 3.A. (\cite{BLin3})} {\it A domain $\om$ is of second kind if and only if $\gamma_{\sscp \om}$ admits at least one
maximum point $q$ such that $D_{\sscp \om}(q)>0$, where,
$$D_{\sscp \om}(q)=\lim_{\eps \to 0^+} \int_{\om\setminus B_{\eps}(q)} {\dsp \frac{\e{8\pi(R(x,q)-\gamma_{\sscp \om}(q)) }-1}{|x-q|^4}}-\int_{\R^2\setminus \om}\frac{1}{|x-q|^4}\,.$$}

\medskip

\brm{\it It is easy to check that the above defined quantity $D_{\sscp \om}(q)$ is proportional to the $D(q)$ as introduced in \eqref{1.5}, 
i.e. $D_{\sscp \om}(q)=c D(q)$ for some constant $c>0$ {provided that $h$ is a constant function and $m=1$. }}
\erm

Next we will need the following particular case of a uniqueness result recently derived in \cite{BJLY}.

\medskip

{\bf Theorem 3.B. (\cite{BJLY})} {\it Let $\om$ be any smooth and bounded domain,  $h$ a constant function and let  $u_{n,i}$, $i=1,2$ be two sequences of solutions of \eqref{1lm} which satisfy, $$\lm_{n,1}=\lm_n=\lm_{n,2},\;n\in \N,$$ and,}
\beq
\label{blowup}
\lm_n\dfrac{\dsp\e{u_n}}{\dsp\ino \e{u_n}}\rightharpoonup 8\pi \dt_{x=q},\mbox{ as }n\to +\ii,
\eeq
{\it weakly in the sense of measures in $\om$, where ${q}$ is a critical point of $\gamma_{\sscp \om}$ such that $\textrm{det}(D^2\gamma_{\sscp \om}({q}))\neq 0$. If $D_{\sscp \om}({q})\neq 0$, then there exists $n_0\in\N$ such that $u_{n,1}=u_{n,2}$ for all $n\ge n_0$.}

\medskip

\brm\label{seckind} {\it A crucial estimate needed to make our argument work has been first obtained in \cite{CCL} for $1$-point blow up solutions
and more recently generalized to the general $m$-point blow up case in \cite{BJLY}. In particular, if $u_n$ satisfies \eqref{blowup}, then it holds,
$$\lm_n-8\pi={c_0}(D_{\sscp \om}(q)+o(1))\left(\max\limits_{x\in \ov{\om}} \dfrac{\dsp \e{u_n(x)}}{\dsp \ino \e{u_n}}\right)^{-1}, \;\ainf,$$
where $c_0>0$ is a constant, see \cite{CCL,BJLY} and Theorem 2.A\mbox{\rm (vi)}.}
\erm

Finally we will need a non degeneracy result about solutions of the Gelfand problem derived in \cite{GlGr}.

\medskip
{\bf {Theorem 3.C. } (\cite{GlGr})} {\it Let $\ue$ be a family of solutions of,}
\beq\label{peps}
\begin{cases}
-\Delta u_{\varepsilon}=\varepsilon^2e^{u_{\varepsilon}}~&{\it in}~\Omega,\\
u_{\varepsilon}=0~&{\it on}~\partial\Omega,
\end{cases}
\eeq {\it which satisfy,}
\beq\label{peps1}
\eps^2\e{\ue}{\rightharpoonup}  8\pi \dt_{q},\,\eps \to 0^+,
\; \mbox{\it weakly in the sense of measures in }\om,
\eeq
{\it where $q$ is  is a critical point of $\gamma_{\sscp \om}$ such that $\textrm{det}(D^2\gamma_{\sscp \om}({q}))\neq 0$. Then there exists $\eps_0>0$ such that the linearized problem relative to \eqref{peps} admits only the trivial solution  for any $\eps\in (0,\eps_0)$.}
\medskip

\section{The proof of Theorem \ref{thm1}}\label{sec6}
In this section we prove Theorem \ref{thm1}.

\medskip

{\it The Proof of Theorem \ref{thm1}.} Let $\pl$ denote any solution of $\prl$. {For a fixed $E_{\sscp \om}>E_{8\pi}$ 
and a fixed function $\lm_{\sscp\ii}:(E_{\sscp \om},+\ii)\to (0,+\ii)$, let $\rh_{\sscp \ii, E}:=\rh_{\sscp \lm_{\sscp \ii}(E)}(\psi_{\sscp \lm_{\sscp \ii}(E)})$ 
be the corresponding density for $E\in(E_{\sscp \om},+\ii)$, where $\psi_{\sscp \lm_{\sscp \ii}(E)}$ is a solution of $\prl$ with
$\lm=\lm_{\sscp \ii}(E)$.} If, for a fixed $E\in(E_{\sscp \om},+\ii)$, we assume that $\pl$ and $\lm_{\sscp\ii}$ are differentiable at the points $\lm=\lm_{\sscp\ii}(E)$ and $E$ respectively, then, by defining,
\beq\label{1.1en}
S_{\ii}(E)=\mathfrak{S}\left(\rh_{\sscp \ii, E}\right),
\eeq
it can be shown that,
\beq\label{diff}
\frac{dS_{\ii}(E)}{d E}=-\lm_{\sscp\ii}(E),
\eeq
see (5.1) in \cite{B2}. Theorem \ref{thm1} follows immediately from the next result which has its own interest, as it describes some properties of the
entropy maximizers in the high energy regime.

\bte\label{thm2} $\mathbf{(a)}$ Under the assumptions of Theorem \ref{thm1}, there exist $E_*> E_{8\pi}$ and a smooth, strictly decreasing function $\lm_{\sscp\ii}(E)$, $E\in[E_*,+\ii)$,
satisfying $\lm_{\sscp\ii}(E)\searrow 8\pi^+$ as $E\to +\ii$,
and there exists a smooth map $(8\pi, \lm_{\sscp \ii}(E_*))\ni \lm \mapsto \psi^{(\ii)}_{\lm}$,
such that $S(E)=S_{\ii}(E)$ takes the form \eqref{1.1en} with
$\psi_{\sscp \lm_{\sscp \ii}(E)}:=\psi^{(\ii)}_{\sscp \lm_{\sscp \ii}(E)}$, $E\in [E_*,+\ii)$.\\
$\mathbf{(b)}$ $S(E)$ is smooth and strictly convex,
and $\frac{dS(E)}{d E}=-\lm_{\sscp\ii}(E)$ for any $E\in [E_*,+\ii)$.
\ete

\brm\label{secep0}
{\it Since $\om$ is convex, then $\gamma_{\sscp \om}$ admits a unique critical point which coincides with its unique
maximum point, say $x_{\sscp \om}\in \om$, which is also non degenerate, that is, $\textrm{det}(D^2\gamma_{\sscp \om}({x_{\sscp \om}}))\neq 0$, see \cite{Gu}.
We will often use this fact throughout the proof.}
\erm

\medskip

{\it The proof of Theorem \ref{thm2}.} {Step 1:} Let $E_n\to +\ii$ and let $\lm_n$ and $\psi_n:=\psi_{\sscp \lm_n}$ be any sequence of entropy maximizers with energy $E_n$. Then, as a consequence of \textbf{(MVP1)}-\textbf{(MVP2)}, $\{\psi_n\}$ solves $\prl$ with
$\lm=\lm_n$ and satisfies \rife{conc}. By the Pohozaev identity jointly with the convexity of $\om$, see for example \cite{clmp1}, we have $\lm_{n}\leq \ov{\lm}<+\ii$ for any $n\in \N$ {and some $\ov{\lm}\in\mathbb{R}$}. Since $\om$ is convex, then a well known moving plane argument shows that $\psi_{n}$ is uniformly bounded from above near $\pa\om$. Therefore, in view of \rife{conc} and the results in \cite{bm} and \cite{yy}, it is not difficult to see that,
\beq\label{lmn8}
\lm_n\to 8\pi \mbox{ and }u_n=\lm_n\psi_n \mbox{ satisfies }\rife{1lm},\rife{blowup}\;\ainf,
\eeq
with blow up point $q=x_{\sscp \om}$ as defined in Remark \ref{secep0}. We point out that here the convexity is used with the Pohozaev identity to provide 
a bound from above for the Lagrange multipliers relative to the high energy regime.
\medskip

{Step 2:} By Remark \ref{secep0} we can apply a result in \cite{suzD}, and conclude that there exists $\eps_*>0$ such that there exists a continuous map from $(0,\eps_*)$ to $C^{2}_0(\ov{\om})$: $\eps\mapsto u_{\eps}$ such that $\ue$ solves \rife{peps} and satisfies \rife{peps1} with $q=x_{\sscp \om}$. Obviously we have,
\beq\label{inj1}
\eps_1\neq \eps_2\quad \Rightarrow \quad u_{\scp \eps_1}\neq u_{\scp \eps_2}.
\eeq

By {Theorem 3.C}, and taking a smaller $\eps_*$ if necessary, then $\ue$ is non degenerate for any $\eps\in (0,\eps_*)$, that is, the linearized problem
relative to \rife{peps} has only the trivial solution. Therefore the implicit function theorem implies that the map $\eps\mapsto \ue$ is smooth in the
$C^{2}_0(\ov{\om})$ topology. Let,
$$
\lep=\eps^2 \ino \e{\ue},\;\eps\in (0,\eps_*),
$$
which therefore is smooth in $(0,\eps_*)$ as well.

\brm\label{secep}
{\it Since $\om$ is of second kind, then Remark \ref{secep0} and Theorem {3.A} imply that $D_{\sscp \om}(x_{\sscp \om})>0$ and so, as a consequence of  Remark \ref{seckind},
we see in particular that $\lm_{\eps}\to 8\pi^+$ as $\eps\searrow 0^+$. This is a crucial point which will play a major role in the rest of the proof.}
\erm

Next, we have the following,
\ble\label{lem1}
By taking a smaller $\eps_*$ if necessary, then the map $\eps\mapsto \lep$ is  injective.
\ele
\proof If not there exist $\eps_{n,i}\searrow 0^+$, $i=1,2$ such that, $\eps_{n,1}\neq\eps_{n,2}$, $\lm_n:=\lm_{\eps_{n,1}}\equiv \lm_{\eps_{n,2}}\to 8\pi$ and
$u_{n,i}:=u_{\scp \eps_{n,i}}$, $i=1,2$ satisfy,
\begin{equation*}
\begin{cases}
\Delta u_{n,i}+\lambda_n\dfrac{\dsp e^{u_{n,i}}}{\dsp \int_{\Omega}e^{u_{n,i}}}=0~&\mathrm{in}~\Omega,\\
u_{n,i}=0~&\mathrm{on}~\partial\Omega,\\
\lambda_n\dfrac{\dsp e^{u_{n,i}}}{\dsp \int_{\Omega}e^{u_{n,i}}}\rightharpoonup 8\pi\delta_{x_{\Omega}}~&\mathrm{as}~n\to+\infty.
\end{cases}
\end{equation*}
In view of  \rife{inj1} we have $u_{n,1}\neq u_{n,2}$. On the other side, by Remark \ref{secep}, we have that
$D_{\sscp \om}(x_{\sscp \om})>0$ and then {Theorem 3.B} yields the desired contradiction.
\finedim

\bigskip

At this point, since $\lep$ is continuous and injective in $(0,\eps_*)$, then it must be strictly monotone in $(0,\eps_*)$ and then, in view of Remark \ref{secep}, we conclude that $\lep$ is strictly increasing. Therefore, it is well defined $\lm_*=\lim_{\eps\to \eps_*^-} \lep$ which in particular is finite, $\lm_*\leq \ov{\lm}$, see step 1. Thus, it is well defined the inverse map $\lm\mapsto \eps_{\scp \lm}$, which is also continuous, strictly increasing and differentiable almost everywhere in $(8\pi,\lm_*)$. As a consequence, by defining,
\beq\label{branch}
\pl^{(\ii)}:=\frac1\lm u_{\scp \eps_{\sscp \lm}},\;\lm \in (8\pi,\lm_*),
\eeq
then the branch,
$$\mathcal{G}_\ii:=\{(\lm,\pl^{(\ii)}),\lm\in (8\pi,\lm_*)\},$$
is a continuous branch of solutions of $\prl$ which also satisfies,
\beq
\label{conc1}
\dfrac{\dsp \e{\lm \pl}}{\dsp \ino \e{\lm \pl}}\rightharpoonup \dt_{ x_{\sscp \sscp \om}},\;\mbox{ as }\lm \searrow 8\pi.
\eeq
Clearly, since $\ue$ is smooth in $\eps$ with respect to the $C^{2}_0(\om)$ topology, then we conclude that $\pl^{(\ii)}$
is differentiable almost everywhere as a function of $\lm$ in $(8\pi,\lm_*)$, with respect to the $C^{2}_0(\om)$ topology as well. 
At this point we can prove that Theorem \ref{th1.1} implies the following fact:
\medskip 

{$\mathbf{(i)}$}  By taking a smaller $\lm_*$ if necessary, then $\pl^{(\ii)}$ is smooth as a function of $\lm$ in $(8\pi,\lm_*)$ with respect to the $C^{2}_0(\om)$ topology.

Indeed, clearly if $\pl^{(\ii)}$ is not differentiable at some point $\lm$, then we must have at least one vanishing eigenvalue in the spectrum of \rife{1.8l}. 
Therefore, if {$\mathbf{(i)}$} were not true, then we could find a bubbling sequence as $\lm_n\to 8\pi$ such that \rife{1.8l} admits a non trivial solution for any $n$, 
which is a contradiction to Theorem \ref{th1.1}. In particular for any $\lm\in (8\pi,\lm_*)$ the linearized problem \rife{1.8l} has only the trivial solution, whence $\lm\mapsto \pl^{(\ii)}$ is smooth as a function of $\lm$ in $(8\pi,\lm_*)$.
\medskip

{Step 3:} We will use  {Theorem 3.B,} \rife{lmn8} in step 1, Remark \ref{secep} and  \rife{branch}, \rife{conc1} to identify
the stream functions corresponding to the entropy maximizing densities at fixed and large enough energy $E$.
More exactly we have,
\ble Let $\ple$ denote the stream function corresponding to any entropy maximizing density at fixed energy $E$, whose Lagrange multiplier is $\lm(E)$.
There exists $E_{*}>E_{8\pi}$ such that, for any
$E\geq E_{*}$,  there exists $\eps(E)\in (0,\eps_*)$ such that,
$$
\lm(E)=\lm_{\sscp \ii}(E):= \left.\lep\right|_{\eps=\eps(E)},\;
\ple=\psi^{(\ii)}_{\sscp \lm_{\sscp\ii}(E)}\equiv \left.\frac1\lep \ue\right|_{\eps=\eps(E)},\mbox{ and in particular } (\lm,\pl^{(\ii)})\in \mathcal{G}_\ii.
$$
\ele
\proof Since $\eps_{\sscp \lm}$, $\lm\in (8\pi,\lm_*)$ is an homeomorphism, then it is enough to check that, by taking a smaller $\lm_*$ if necessary,
there exists $E_*>E_{8\pi}$ such that, for any $E>E_*$, there exists $\lm\in (8\pi,\lm_*)$ such that
the stream function $\ple$ corresponding to any entropy maximizing density at fixed $E$ satisfies,
$$
\ple\equiv \pl^{(\ii)}=\frac1\lm u_{\scp \eps_{\sscp \lm}},\quad(\lm,\pl^{(\ii)})\in \mathcal{G}_\ii.
$$
By ${\bf (MVP1)}$ we know that $(\lm(E), \ple)$ exists for any $E>0$. We argue by contradiction and suppose that there exists $E_n\to +\ii$ and $\lm_n:=\lm(E_n)$ such that $\psi_n:=\psi_{\sscp \lm_n}$, $n\in \N$, are the stream
functions of some entropy maximizing densities at energies $E_n$ but $(\lm_n,\psi_n)\notin \mathcal{G}_\ii$ for any $n$.
In view of \rife{lmn8} in step 1 and Remark \ref{secep} we see that $\lm_n\to 8\pi^+$ and $v_n=\lm_n\psi_n$ satisfies \rife{blowup} with $q=x_{\sscp \om}$.
Therefore, for any $n$ large enough, we also find that $\lm_n\in (8\pi,\lm_*)$ and we denote by $u_n=\lm_n \psi^{(\ii)}_{\sscp \lm_n}$ the sequence of solutions
of \rife{1lm}, $(\lm_n,\psi^{(\ii)}_{\sscp \lm_n})\subset \mathcal{G}_\ii$, as defined in \rife{branch}, which in view of \rife{conc1} also satisfy \rife{blowup}.
Therefore we have found two distinct sequences of solutions of \rife{1lm} sharing
the same $\lm_n$ and satisfying \rife{blowup} which is a contradiction to {Theorem 3.B.}
\finedim

\medskip

Next we claim that:
\medskip

{$\mathbf{(ii)}$} By taking a {larger} $E_*$ if necessary, then $\lm_{\sscp\ii}(E)$ is  smooth, strictly decreasing in $[E_*,+\ii)$, and $\lm_{\sscp\ii}(E)\searrow 8\pi^+$ as $E\to +\ii$.\\

{\it Proof of $\mathbf{(ii)}$.} Since $\lm_\eps$ is smooth and strictly increasing in $(0,\eps_*)$ and $\lep\to 8\pi^+$ as $\eps\to 0^+$, then it is enough to show that the energy $E_{\scp \eps}$ is smooth as a function of $\eps$, with $\frac{d E_\eps}{d\eps}<0$ in $(0,\eps_*)$ and $E_{\scp \eps}\to +\ii$ as $\eps\searrow 0^+$. This is the content of the following,

\ble
\label{enrgeps}
By taking a smaller $\eps_*$ if necessary, then, for $\eps\in (0,\eps_*)$, is well defined the map,
$$E_{\scp \eps}:=\mathcal{E}(\mbox{\rm $\rh_{\sscp \lep}(\psi_{\scp \lep})$}),$$
which is also smooth and strictly decreasing, with $\frac{d E_\eps}{d\eps}\to -\ii$ and $E_{\scp \eps}\to +\ii$ as $\eps\searrow 0^+$.
\ele

\proof Let $\we=\frac{d\ue}{d\eps}$ which is well defined for $\eps\in (0,\eps_*)$ for some $\eps_*$ small enough by { Theorem 3.C.}
It is easy to check that,
\beq\label{ue1}
\begin{cases}
\Delta w_{(\eps)}+2\eps e^{u_{\eps}}+\eps^2e^{u_{\eps}}w_{(\eps)}=0~&\mathrm{in}~\Omega,\\
w_{(\eps)}=0~&\mathrm{on}~\partial\Omega,
\end{cases}
\eeq
and
\beq\label{ue2}
\frac{d\lep}{d\eps}=2\eps \ino \e{\ue}+\eps^2\ino \e{\ue}\we,
\eeq
and in particular that,
$$
E_\eps=\dfrac 12\dfrac{\eps^2\ino \e{\ue}\ue}{\lep^2},
$$
where we used
\begin{equation*}
u_{\eps}=\lambda_{\eps}\psi_{\eps}\quad\mathrm{and}\quad\lambda_{\eps}=\eps^2\int_{\Omega}e^{u_{\eps}}.
\end{equation*}
Multiplying the equation  in \rife{ue1} by $\ue$, integrating by parts and using \rife{peps} we find that,
$$
\eps^2 \ino \e{\ue} \we=2\eps \ino \e{\ue}\ue + \eps^2\ino \e{\ue}\we \ue,
$$
which readily implies that,
\beq\label{ue4}
\frac{d}{d\eps}\left(\eps^2\ino \e{\ue}\ue\right)=2\eps \ino \e{\ue}\ue + \eps^2\ino \e{\ue}\we \ue+\eps^2 \ino \e{\ue} \we=2\eps^2 \ino \e{\ue} \we.
\eeq
In view of \rife{ue2} and \rife{ue4} we find that,
\begin{align*}
\lep^4\frac{d E_\eps}{d\eps}
=~&\frac12 \frac{d}{d\eps}\left(\eps^2\ino \e{\ue}\ue\right)\lep^2-\lep \frac{d\lep}{d\eps}\eps^2\ino \e{\ue}\ue\\
=~&\left(\eps^2 \ino \e{\ue} \we\right)\lep^2-\lep \frac{d\lep}{d\eps}\eps^2\ino \e{\ue}\ue\\
=~&\left(\frac{d\lep}{d\eps}-\frac{2\lep}{\eps}\right)\lep^2-\lep \frac{d\lep}{d\eps}\eps^2\ino \e{\ue}\ue
=\frac{d\lep}{d\eps}\left( \lep^2- \lep \eps^2\ino \e{\ue}\ue\right)-\frac{2\lep^3}{\eps}.
\end{align*}
At this point, by using the fact that $\frac{d\lep}{d\eps}\geq 0$,  $\eps^2\ino \e{\ue}\ue\to +\ii$ and that $\lep\to 8\pi$ as $\eps\to 0^+$, 
then we conclude by the last equality that $\frac{d E_\eps}{d\eps}\to -\ii$ as $\eps\to 0^+$. 
The fact that $\eps^2\ino \e{\ue}\ue\to +\ii$ is a straightforward consequence of well known blow up arguments, see for example
\cite{yy}, which however could be easily deduced from \rife{2.5} and Theorem 2.A. Clearly {\it $\mathbf{(i)}$} and {\it $\mathbf{(ii)}$} imply that $\lm_{\sscp \ii}(E)$ and $\psi^{(\ii)}_{\sscp \lm_{\sscp\ii}(E)}$ satisfy all the properties needed in $[E_*,+\ii)$ and in $(8\pi,\lm_{\sscp \ii}(E_*))$ respectively, which concludes the proof of $\mathbf{(a)}$.
\medskip

The conclusion $\mathbf{(b)}$ is a straightforward consequence of $\mathbf{(a)}$ and \rife{diff}.
\finedim
\medskip

\section{Appendix: the proof of $b_0=b_{j,i}=0$}\label{appendix}
In this section, we are going to show $b_0=b_{j,i}=0$ by making use of the Pohozaev type identities \eqref{4.3} and \eqref{4.47}. 
The argument is an adaptation of those in \cite{ly2} and \cite{BJLY}. Therefore we shall discuss the main steps and refer the readers to \cite{ly2,BJLY} 
for further details.
\medskip

To analyse the various terms in \eqref{4.3} and \eqref{4.47} we need a delicate estimate about $\nabla v_{n,j}$ and $\nabla\zeta_n$. Let
\begin{equation}
\label{4.5}
\widetilde{G}(x)=\frac{\lambda_n}{m}\sum_{l=1}^m G(x,x_{n,l}){.}
\end{equation}{
In view of Theorem 2.A, it is not difficult to see that,}
\begin{equation}
\label{4.7}
\begin{aligned}
\nabla v_{n,j}=\nabla(\widetilde{G}-\phi_{n,j})+o(e^{-\frac{\mu_{n,j}}{2}})
{=-4\frac{x-x_{n,j}}{|x-x_{n,j}|^2}+ o(e^{-\frac{\mu_{n,j}}{2}})}~\mathrm{in}~\bigcup_{l=1}^mB_{2r_0}({x}_{n,l})\setminus B_{r/2}({x}_{n,l}).
\end{aligned}
\end{equation}
{For the estimate of $\nabla\zeta_n$ we will apply the Green representation formula and a suitable scaling. 
It is convenient to introduce the following notations,}
\begin{equation}
\label{3.6}
\Lambda_{n,j,r}^{-}=re^{-\mu_{n,j}/2},\quad \Lambda_{n,j,r}^{+}=re^{\mu_{n,j}/2},
\ \ \textrm{for any}\ \  r>0, \end{equation}
and
\begin{equation}
\label{sc}\overline{f}(z)=f(e^{-\frac{\mu_{n,j}}{2}}z+x_{n,j}),~|z|<r_0 e^{\frac{\mu_{n,j}}{2}}
~\mbox{for any function}~f:B_{r_0}({x}_{n,j})\to\mathbb{R}.
\end{equation}
\medskip

\begin{lemma} 
\label{le3.2}
\begin{equation}
\label{3.7}
\begin{aligned}
\zeta_n(x)-\zeta_{n,0}=&\sum_{j=1}^mA_{n,j}G(x_{n,j},x)+\sum_{j=1}^m\sum_{h=1}^2B_{n,j,h}\partial_{y_h}G(y,x)|_{y=x_{n,j}}
+o(e^{-\frac{1}{2}\mu_{n,1}}),
\end{aligned}
\end{equation}
holds in $C^1(\ov{\om}\setminus\cup_{j=1}^m B_{\delta}(x_{n,j}))$,
for a suitably defined positive constant $\delta>0$, where $\zeta_{n,0}=\zeta_n|_{\partial\Omega}$,
$\partial_{y_h} G({y},{x})=\frac{\partial G({y},{x})}{\partial {y}_h}$, $y=(y_1,y_2)$,
\begin{equation*}
A_{n,j}=\int_{\Omega_j}f_n^*(y)\mathrm{d}y,
\quad\mbox{and}\quad
B_{n,j,h}=e^{-\frac{1}{2}\mu_{n,j}}\frac{b_{j,h}4\sqrt{8}}{\sqrt{\lambda_n h(q_j)}}
\int_{\mathbb{R}^2}\frac{|z|^2}{(1+|z|^2)^3}\mathrm{d}z.
\end{equation*}
Moreover, there exists a constant $C>0$ such that,
\begin{equation}
\label{3.8}
\begin{aligned}
\left|\zeta_n(x)-\zeta_{n,0}-\sum_{j=1}^m A_{n,j}G(x_{n,j},x)\right| \leq {C}\sum_{j=1}^m e^{-\frac{\mu_{n,j}}{2}}\frac{1}{|x-x_{n,j}|},
\end{aligned}
\end{equation}
for $x\in \Omega\setminus\bigcup_{j=1}^m B_{\Lambda_{n,j,R}^{-}}(x_{n,j})$, for any $R>0$ (where $C$ do not depend by $R$).
\end{lemma}

\begin{proof}{The proof of Lemma \ref{le3.2} is based on the Green representation formula}
\begin{equation}
\label{3.9}
\begin{aligned}
\zeta_n(x)-\zeta_{n,0}&=\int_{\Omega}G(y,x)f_n^*(y)\mathrm{d}y\\
&=\sum_{j=1}^m A_{n,j}G(x_{n,j},x)+\sum_{j=1}^m\int_{\Omega_j}(G(y,x)-G(x_{n,j},x))f_n^*(y)\mathrm{d}y{,}
\end{aligned}
\end{equation}
and the estimates of $\tilde{u}_n$ in Theorem 2.A. We refer to \cite[Lemma 2.3]{ly2} and  \cite[Lemma 3.3]{BJLY} for further details. \end{proof}

\begin{lemma}
\label{le4.2}
\begin{equation*}
\begin{aligned}
\mathrm{L.H.S.~of}~\eqref{4.3}
=& -4A_{n,j}-\frac{256b_0e^{-\mu_{n,1}}h(q_j)e^{G_j^*(q_j)}}{\lambda_n(h(q_1))^2e^{G_1^*(q_1)}}
\int_{\Omega_j\setminus B_{r}(q_j)}e^{\Phi_j(x,\mathbf{q})}\mathrm{d}x\\ & +o(e^{-\frac{\mu_{n,j}}{2}}\sum_{l=1}^m|A_{n,l}|)+o(e^{-\mu_{n,j}}),
\end{aligned}
\end{equation*}
for fixed $r\in (0,r_0)$, with $r_0$ as defined right below \eqref{1.4}.
\end{lemma}
\begin{proof} We shall first derive a refined estimate about  $\nabla\zeta_n$. By the Green representation formula, we see that,
\begin{equation}
\label{4.10}
\begin{aligned}
\zeta_n(x)-\zeta_{n,0}
=\int_{\Omega}G(y,x)f_n^*(y)\mathrm{d}y
&=\sum_{l=1}^mA_{n,l}G(x_{n,l},x)+\sum_{l=1}^m\sum_{h=1}^2B_{n,l,h}\left(\partial_{y_h}G(y,x)|_{y=x_{n,l}}\right)\\
&\quad+\frac{1}{2}\sum_{l=1}^m\sum_{h,k=1}^2C_{n,l,h,k}\left(\partial^2_{y_hy_k}G(y,x)|_{y=x_{n,l}}\right)\\
&\quad+\sum_{l=1}^m\int_{\Omega_l}\Psi_{n,l}(y,x)f_n^*(y)\mathrm{d}y,
\end{aligned}
\end{equation}
where
\begin{align*}
A_{n,l}=&\int_{\Omega_{l}}f_n^*(y)\mathrm{d}y,~B_{n,l,h}=\int_{B_{r_0}(x_{n,l})}(y-x_{n,l})_h f_n^*(y)\mathrm{d}{y},
\\\quad &{C_{n,l,h,k}}=\int_{B_{r_0}(x_{n,l})}(y-x_{n,l})_h(y-x_{n,l})_{k}f_n^*({y})\mathrm{d}{y},
\end{align*}
and
\begin{align*}
\Psi_{n,l}(y,x)=~&G(y,x)-G(x_{n,l},x)-\langle\partial_{y}G(y,x)|_{y=x_{n,l}},y-x_{n,l}\rangle\\
&-\frac{1}{2}\langle\partial^2_yG({y},{x})|_{y=x_{n,l}}(y-x_{n,l}),y-x_{n,l}\rangle.
\end{align*}

Fix $\bar r\in(0,\frac{r}{2})$. By using \eqref{farblow} and Theorem 2.A, we see that,
\begin{equation}
\begin{aligned}
\label{4.13}
f_n^*{(y)}=\frac{64 e^{-\mu_{n,j}}}{\lambda_nh(x_{n,j})}{e^{\Phi_j(y,\mathbf{x}_n)}}(-b_0+o(1))~\mathrm{for}~
y\in\ov{\om_j}\setminus B_{\bar r}(x_{n,j}),
\end{aligned}
\end{equation}
where $\mathbf{x}_n=(x_{n,1},\cdots,x_{n,m})$ and
$$\Phi_j(y,\mathbf{x}_n)=\sum_{l=1}^m 8\pi G(y,x_{n,l})-G_j^*(x_{n,j})+\log h(y)-\log h(x_{n,j}).$$

Let us define,
\begin{equation}
\label{4.16}
\begin{aligned}
\overline{G}_n(x)=~&\zeta_{n,0}+\sum_{l=1}^mA_{n,l}G(x_{n,l},x)
+\sum_{l=1}^m\sum_{h=1}^2B_{n,l,h}\partial_{y_h}G(y,x)|_{y=x_{n,l}}\\
&+\frac{1}{2}\sum_{l=1}^m\sum_{h,k=1}^2C_{n,l,h,k}\partial^2_{y_hy_k}G(y,x)|_{y=x_{n,l}},
\end{aligned}
\end{equation}
and
\begin{equation}
\label{4.18}
\zeta_n^*(x)=-b_0\sum_{l=1}^m\int_{\Omega_l\setminus B_{\bar r}(x_{n,l})}
\frac{64e^{-\mu_{n,l}}\Psi_{n,l}(y,x)}{\lambda_nh(x_{n,l})}e^{\Phi_l({y},\mathbf{x}_n)}\mathrm{d}y.
\end{equation}
By  \eqref{4.10} and \eqref{4.13}, we conclude that for $x\in\partial B_r(x_{n,j})$, it holds,
\begin{equation}
\label{4.17}
\begin{aligned}
 \zeta_n(x)-\overline{G}_n(x)
 &{
=} \zeta_n^*(x)
+O\left(\frac{\bar re^{-\mu_{n,j}}}{|x-x_{n,j}|^3}\right)  +o(e^{-\mu_{n,j}})~\mathrm{in}~C^1(\partial B_r(x_{n,j})).
\end{aligned}\end{equation}
Substituting \eqref{4.7} and \eqref{4.17} into \eqref{4.3}, we see that,
\begin{equation}
\label{4.19}
\begin{aligned}
\mathrm{L.H.S.~of}~\eqref{4.3}=~&\int_{\partial B_r(x_{n,j})}4\langle\nu,D(\overline{G}_n+\zeta_n^*)(x)\rangle\mathrm{d}\sigma(x)
+o(e^{-\frac{\mu_{n,j}}{2}}\sum_{l=1}^m|A_{n,l}|)\\
&+O\left(\frac{{\bar{r}}e^{-\mu_{n,j}}}{r^3}\right)+o(e^{-\mu_{n,j}}),
\end{aligned}
\end{equation}
for any $\bar r\in(0,\frac{r}{2}).$ The estimate about the right hand side of \eqref{4.19} depends on the following identity
\begin{equation}
\label{4.20}
\begin{aligned}
&\Delta u\left(\nabla v\cdot(x-x_{n,j})\right)+\Delta v\left(\nabla u\cdot(x-x_{n,j})\right)\\
&=\mathrm{div}\left(\nabla u(\nabla v\cdot(x-x_{n,j}))
+\nabla v(\nabla u\cdot(x-x_{n,j}))-\nabla u\cdot\nabla v(x-x_{n,j})\right).
\end{aligned}
\end{equation}
Letting $u=\overline{G}_n$ and $v=\widetilde{G}-\phi_{n,j}$ (resp. $u=\Psi_{n,j}(y,\cdot)$ and $v=\widetilde{G}-\phi_{n,j}$), 
then we are able to compute the right hand side of \eqref{4.19} 
$\int_{\partial B_r(x_{n,j})} \langle\nu,D \overline{G}_n \rangle\mathrm{d}\sigma $ 
(resp. $\int_{\partial B_r(x_{n,j})} \langle\nu,D \zeta_n^* \rangle\mathrm{d}\sigma$). See \cite{ly2} and \cite[Lemma 4.2]{BJLY} for further details.
\end{proof}

The following result is needed to estimate the right hand side of \eqref{4.3}.

\begin{lemma}
\label{le4.3}
\begin{equation*}
\begin{aligned}
(i)~\int_{\partial B_r(x_{n,j})}rf_n^*\mathrm{d}\sigma
=&-\frac{32\pi e^{-\mu_{n,1}}b_0h(q_j)e^{G_j^*(q_j)}(\Delta\log h(q_j))}{\lambda_n(h({q}_1))^2e^{G_1^*(q_1)}}\\
&-\frac{128e^{-\mu_{n,1}}b_0h(q_j)e^{G_j^*(q_j)}}{\lambda_n(h(q_1))^2e^{G_1^*(q_1)}}\frac{\pi}{r^2}
+O(re^{-\mu_{n,1}})+\frac{o(e^{-\mu_{n,1}})}{r^2}.
\end{aligned}
\end{equation*}
\begin{equation*}
(ii)~\sum_{j=1}^m\int_{B_r(x_{n,j})}f_n^*\mathrm{d}x
=\frac{64b_0e^{-\mu_{n,1}}}{\lambda_n(h(q_1))^2e^{G_1^*(q_1)}}\sum_{j=1}^m
\int_{\Omega_j\setminus B_r(q_j)}h(q_j)e^{G_j^*(q_j)}e^{\Phi_j(x,\mathbf{q})}\mathrm{d}x+o(e^{-\mu_{n,1}}).
\end{equation*}
\begin{equation*}
\begin{aligned}
(iii)~&\int_{B_r(x_{n,j})}f_n^*\langle D(\log h+\phi_{n,j}),x-x_{n,j}\rangle\mathrm{d}{x}\\
&=\Bigg[\frac{32\pi\zeta_{n,0}(\Delta\log h(q_j))}{\lambda_n(h(q_1))^2e^{G_{1}^*(q_1)}}h(q_j)e^{G_j^*(q_j)}e^{-\mu_{n,1}}\Bigg(\mu_{n,1}
+\log\Big(\frac{\lambda_n(h({q}_1))^2e^{G_1^*(q_1)}}{8h(q_j)e^{G_j^*(q_j)}}r^2\Big)-2\Bigg)\Bigg]\\
&\quad+O(1)(\frac{|\log r|e^{-\mu_{n,j}}}{(\mu_{n,j})^2})+O(re^{-\mu_{n,j}})+o(e^{-\mu_{n,j}})|\log R|+O\left(\frac{e^{-2\mu_{n,j}}}{r^2}\right)+O\left(\frac{e^{-\mu_{n,j}}}{R^2}\right)\\
&\quad+O(1)\Bigg(\sum_{l=1}^m(|A_{n,l}|+e^{-\frac{\mu_{n,j}}{2}})\Big(\frac{e^{-\frac{\mu_{n,j}}{2}}}{R}+e^{-\mu_{n,j}}
(\mu_{n,j}+|\log r|)\Big)\Bigg)~\mbox{for any}~R>1,
\end{aligned}
\end{equation*}
where $O(1)$  denotes any quantity uniformly bounded with respect to $r, R$ and $n$.
\end{lemma}
\begin{proof} The proof of Lemma \ref{le4.3} is based on Theorem 2.A  and  $\int_{\Omega}f^*_ndy=0$ as in \cite{ly2,BJLY}. 
We sketch the proof of $(iii)$ for reader's convenience. By Theorem 2.A, the expansion of $D(\log h+\phi_{n,j})$, 
and  the scaling $x=e^{-\frac{\mu_{n,j}}{2}}z+x_{n,j}$,   we see that,
\begin{equation}
\label{4.39}
\begin{aligned}
&\int_{B_r(x_{n,j})}f_n^*\langle D(\log h+\phi_{n,j}),x-x_{n,j}\rangle\d x\\
&=\int_{B_{\Lambda_{n,j,r}^{+}}(0)}\frac{\lambda_n h(x_{n,j})\overline{\zeta_n}}{(1+\frac{\lambda_nh(x_{n,j})}{8}|z|^2)^2}
\langle D^2(\log h+\phi_{n,j})(x_{n,j})z,z\rangle e^{-\mu_{n,j}}\d z\\
&\quad+O(1)\left(\frac{|\log r|e^{-\mu_{n,j}}}{(\mu_{n,j})^2}\right)+O(re^{-\mu_{n,j}})+o(e^{-\mu_{n,j}}).
\end{aligned}\end{equation}
At this point, by using the following identity $\int\frac{(1-r^2)r^3}{(1+r^2)^3}\mathrm{d}r=\frac{1}{2}\left(\frac{-3r^2-2}{(1+r^2)^2}-\log(1+r^2)\right)+C,$
together with \eqref{nearblow} and \eqref{farblow}, then for any fixed and large $R>0$ we see that,
\begin{align*}
&\int_{B_R(0)}\frac{\lambda_nh(x_{n,j})\overline{\zeta_n}}{(1+\frac{\lambda_nh(x_{n,j})}{8}|z|^2)^2}
\langle D^2(\log h+\phi_{n,j})(x_{n,j})z,z\rangle e^{-\mu_{n,j}}\d z\\
&=\frac{32\pi b_0\Delta(\log h+\phi_{n,j})(x_{n,j})}{\lambda_n h(x_{n,j})}e^{-\mu_{n,j}}
\Big(\frac{\frac{\lambda_n h(x_{n,j})}{8}R^2(1+\frac{\lambda_n h(x_{n,j})}{4}R^2)}{(1+\frac{\lambda_n h(x_{n,j})}{8}R^2)^2}
-\log(1+\frac{\lambda_nh(x_{n,j})}{8}R^2)\Big)\\
&\quad+o(e^{-\mu_{n,j}})|\log R|.
\end{align*}
Next, let us observe that,
\begin{equation}
\label{4.40}
\int\frac{r^3}{(1+r^2)^2}\d r=\frac{1}{2}\Big(\frac{1}{1+r^2}+\log(1+r^2)\Big)+C.
\end{equation}
In view of \eqref{3.8}, we also see that if $|z|\ge R$ then it holds,
\begin{equation}
\label{4.41}
\overline{\zeta_n}(z)=\zeta_{n,0}+O(1)\Big(\sum_{l=1}^m(|A_{n,l}|+e^{-\frac{\mu_{n,j}}{2}})(\frac{e^{\frac{\mu_{n,j}}{2}}}{|z|}+1)\Big).
\end{equation}
Since $\zeta_{n,0}$ is constant, we also conclude from \eqref{4.40} and \eqref{4.41} that,
{\allowdisplaybreaks
\begin{align*}
&\int_{B_{\Lambda_{n,j,r}^{+}}(0)\setminus B_R(0)}\frac{\lambda_nh(x_{n,j})\overline{\zeta_n}}{(1+\frac{\lambda_nh(x_{n,j})}{8}|z|^2)^2}
\langle D^2(\log h+\phi_{n,j})(x_{n,j})z,z\rangle e^{-\mu_{n,j}}\d z
\\
=&-\left[\frac{1}{1+\frac{\lambda_nh(x_{n,j})}{8}R^2}+\log(1+\frac{\lambda_nh(x_{n,j})}{8}R^2)
-\mu_{n,j}-\log(\frac{\lambda_nh(x_{n,j})}{8}r^2)\right]\\
&\times\frac{32\pi\zeta_{n,0}\Delta(\log h+\phi_{n,j})(x_{n,j})}{\lambda_nh(x_{n,j})}e^{-\mu_{n,j}}+O(\frac{e^{-2\mu_{n,j}}}{r^2})\\&+O(1)\sum_{l=1}^m(|A_{n,l}|+e^{-\frac{\mu_{n,j}}{2}})\Big(\frac{e^{-\frac{\mu_{n,j}}{2}}}{R}+e^{-\mu_{n,j}}(\mu_{n,j}+|\log r|)\Big).
\end{align*}}
By  \eqref{farblow}, it is easy to check that,
\begin{equation}
\label{4.43}
\zeta_{n,0}=-b_0+o(1).
\end{equation}
The estimates \eqref{4.39}-\eqref{4.43} used together with Theorem 2.A conclude the proof of Lemma~\ref{le4.3}-(iii).
\end{proof}

Now we are able to prove that $b_{j,0}=0$ for all $j=1,\cdots,m$.

\begin{lemma}
\label{le4.4}$ $\\
$(i)$ $A_{n,j}=\int_{\Omega_j}f_n^*(y)\d y= o(e^{-\frac{\mu_{n,j}}{2}})$.\\
$(ii)$ $b_0=0$ and in particular $b_{j,0}=0$, $j=1,\cdots,m$.
\end{lemma}
\begin{proof}
$(i)$ This is an immediate consequence of Theorem 2.A and Lemmas \ref{le4.2}-\ref{le4.3}.

\noindent $(ii)$  For any $r>0$, let us set,
\begin{equation}
\label{4.45}
r_j=r\sqrt{8h(q_j)G_j(q_j)}\quad \mathrm{for}\quad j=1,\cdots,m.
\end{equation}
Note that $\sum_{j=1}^m A_{n,j}=0$. This fact, together with the Pohozaev type indentity \eqref{4.3}, Lemmas \ref{le4.2}-\ref{le4.3}, and $(i)$, implies that,
\begin{equation}
\label{4.46}
\begin{aligned}
&-\frac{256b_0e^{-\mu_{n,1}}}{\lambda_n(h(q_1))^2e^{G_1^*(q_1)}}\sum_{j=1}^mh(q_j)e^{G_j^*(q_j)}
\int_{\Omega_j\setminus B_{r_j}(q_j)}e^{\Phi_j(y,\mathbf{q})}\d y\\
&=-\frac{128e^{-\mu_{n,1}}b_0}{\lambda_n(h(q_1))^2e^{G_1^*(q_1)}}\sum_{j=1}^mh(q_j)e^{G_j^*(q_j)}\frac{\pi}{r_j^2}
-\frac{32\pi e^{-\mu_{n,1}}b_0\ell(\mathbf{q})}{\lambda_n(h(q_1))^2e^{G_1^*(q_1)}}\\
&\quad-\frac{128e^{-\mu_{n,1}}b_0}{\lambda_n(h(q_1))^2e^{G_1^*(q_1)}}\sum_{j=1}^mh(q_j)e^{G_j^*(q_j)}
\int_{\Omega_j\setminus B_{r_j}(q_j)}e^{\Phi_j(y,\mathbf{q})}\d y\\
&\quad-\frac{32\pi\zeta_{n,0}\ell(\mathbf{q})e^{-\mu_{n,1}}}{\lambda_n(h(q_1))^2e^{G_1^*(q_1)}}
\Big(\mu_{n,1}+\log\Big(\lambda_n(h(q_1))^2e^{{G_1^*}(q_1)}r^2\Big)-2\Big)\\
&\quad+O(e^{-\mu_{n,1}})(r+R^{-1})+o(e^{-\mu_{n,1}})(\frac{1}{r^2}+\log R).
\end{aligned}
\end{equation}
If  either $\ell(\mathbf{q})\neq0$ or $D(\mathbf{q})\neq0$, then \eqref{4.46} implies  $b_0=0$.
In view of \eqref{nearblow},  we also obtain $b_{j,0}=b_0=0$, $j=1,\cdots, m$. This fact concludes the proof of $(ii)$.
\end{proof}

Next, by using \eqref{4.47}, we shall  prove that $b_{j,1}=b_{j,2}=0$. 
\begin{lemma}
\label{le4.6}
\begin{equation*}
\mathrm{R.H.S.~of}~\eqref{4.47}=\tilde{B}_j\Big(\sum_{h=1}^2D^2_{hi}(\phi_{n,j}+\log h)(x_{n,j})e^{-\frac{\mu_{n,j}}{2}}b_{j,h}\Big)
+o(e^{-\frac{\mu_{n,j}}{2}}),
\end{equation*}
where $$\tilde{B}_j=4\sqrt{\frac{8}{\lambda_nh(x_{n,j})}}\int_{\mathbb{R}^2}\frac{|z|^2}{(1+|z|^2)^3}\mathrm{d} z.$$
\end{lemma}
\begin{proof} Since $\mathbf{q}$ is a critical point of $f_m$, then by using Theorem 2.A,  we find that,
\begin{equation}
\label{4.50}
\begin{aligned}
D_i(\phi_{n,j}+\log h)(x_{n,j})=D_i(G_j^*+\log h)(x_{n,j})+O(\mu_{n,j}e^{-\mu_{n,j}})
=O(\mu_{n,j}e^{-\mu_{n,j}}).
\end{aligned}
\end{equation}
As a consequence, in view of the blow up profile of $\tilde{u}_n$ in Theorem 2.A and \eqref{nearblow}, we conclude that,
\begin{equation}
\label{4.48}
\begin{aligned}
&\mathrm{R.H.S.~of}~\eqref{4.47}\\&=-\int_{\partial B_r(x_{n,j})}\lambda_nh(x)e^{\tilde{u}_n}\zeta_n\frac{(x-x_{n,j})_i}{|x-x_{n,j}|}\mathrm{d}\sigma+\int_{B_r(x_{n,j})}\lambda_nh(x)e^{\tilde{u}_n}\zeta_nD_i(\phi_{n,j}+\log h)\d x
\\
&=4\sqrt{\frac{8}{\lambda_nh(x_{n,j})}}\int_{\mathbb{R}^2}\frac{|z|^2}{(1+|z|^2)^3}\d z\Big( \sum_{h=1}^2D^2_{hi}(\phi_{n,j}+\log h)(x_{n,j})e^{-\frac{\mu_{n,j}}{2}}b_{j,h}\Big)+o(e^{-\frac{\mu_{n,j}}{2}}),
\end{aligned}
\end{equation}which implies that the conclusion of Lemma \ref{le4.6} holds.
\end{proof}

\begin{lemma}
\label{le4.7}
\begin{equation*}
\begin{aligned}
\mathrm{L.H.S.~of}~\eqref{4.47}
=&-8\pi\Bigg[\sum_{l\neq j}e^{-\frac{\mu_{n,l}}{2}}D_iG_{n,l}^*(x_{n,j})+ e^{-\frac{\mu_{n,j}}{2}}D_i\sum_{h=1}^2\partial_{y_h}R(y,x)|_{x=y=x_{n,j}}b_{j,h}\tilde B_j\Bigg]\\
&+o(e^{-\frac{\mu_{n,j}}{2}}),
\end{aligned}
\end{equation*}
where $G_{n,l}^*(x)=\sum_{h=1}^2{\partial_{y_h}}G(y,x)|_{y=x_{n,l}}b_{l,h}\tilde{B}_l.$
\end{lemma}
\begin{proof}
By the definition of $G_{n,i}^*$, we have for any $\theta\in(0,r)$,
$$\Delta G_{n,l}^*=0~\mathrm{in}~B_r(x_{n,j})\setminus B_{\theta}(x_{n,j}).$$
Then for $x\in B_r(x_{n,j})\setminus B_{\theta}(x_{n,j})$, and setting $e_i=\frac{x_i}{|x|}$, $i=1,2$, we have,
\begin{equation*}
\begin{aligned}
0=~&\Delta G_{n,l}^*D_i\log\frac{1}{|x-x_{n,j}|}+\Delta\log\frac{1}{|x-x_{n,j}|}D_iG_{n,l}^*\\
=~&\mathrm{div}\Bigg(\nabla G_{n,l}^*D_i\log\frac{1}{|x-x_{n,j}|}+\nabla\log\frac{1}{|x-x_{n,j}|}D_iG_{n,l}^*
-\nabla G_{n,l}^*\cdot\nabla\log\frac{1}{|x-x_{n,j}|}e_i\Bigg),
\end{aligned}
\end{equation*}
which readily implies that,
\begin{equation}
\label{4.52}
\begin{aligned}
\int_{\partial B_r(x_{n,j})}\frac{\nabla_iG_{n,l}^*}{|x-x_{n,j}|}\d\sigma
&=\int_{\partial B_\theta(x_{n,j})}\frac{\nabla_iG_{n,l}^*}{|x-x_{n,j}|}\mathrm{d}\sigma.
\end{aligned}
\end{equation}
In view of  Lemma \ref{le3.2} and Lemma \ref{le4.4}, we also have,
\begin{equation}
\label{4.53}
\zeta_n(x)-\zeta_{n,0}=\sum_{l=1}^me^{-\frac{\mu_{n,l}}{2}}G_{n,l}^*(x) +o(e^{-\frac{\mu_{n,j}}{2}})\quad\mathrm{in}\quad C^1(B_r(x_{n,j})\setminus B_{\theta}(x_{n,j})).
\end{equation}
By using $D_iD_h(\log|z|)=\frac{\delta_{ih}|z|^2-2z_iz_h}{|z|^4},$ we see that,
\begin{equation*}
\begin{aligned}
\int_{\partial B_{\theta}(x_{n,j})}\frac{\nabla_iG_{n,j}^*}{|x-x_{n,j}|}\mathrm{d}\sigma
=2\pi D_i\sum_{h=1}^2\partial_{y_h}R(y,x)|_{x=y=x_{n,j}}b_{j,h}\tilde{B}_j+o_\theta(1).
\end{aligned}
\end{equation*}
This fact, together with \eqref{4.52}-\eqref{4.53} and  the estimate about $\nabla v_{n,j}$ in \eqref{4.7},  we see that, for any $\theta\in(0,r)$,
\begin{equation*}
\begin{aligned}
\mathrm{L.H.S.~ of}~\eqref{4.47}=&-4\int_{\partial B_r(x_{n,j})}
\sum_{l=1}^me^{-\frac{\mu_{n,l}}{2}}\frac{\nabla_iG_{n,l}^*}{|x-x_{n,j}|}\d\sigma+o(e^{-\frac{\mu_{n,j}}{2}})\\
=&-8\pi\Bigg[\sum_{l\neq j} e^{-\frac{\mu_{n,l}}{2}}D_iG_{n,l}^*(x_{n,j})
+e^{-\frac{\mu_{n,j}}{2}}D_i\sum_{h=1}^2\partial_{y_h}R({y},{x})|_{x=y=x_{n,j}}b_{j,h}\tilde{B}_j\Bigg]\\
&+o(e^{-\frac{\mu_{n,j}}{2}})+o_\theta(1) e^{-\frac{\mu_{n,j}}{2}},
\end{aligned}
\end{equation*}
which proves the claim of Lemma \ref{le4.7}.
\end{proof}

Finally, we have the following,
\begin{lemma}
\label{le4.8}
$b_{j,1}=b_{j,2}=0$, $j=1,\cdots, m$.
\end{lemma}

\begin{proof}
From the Pohozaev type indentitly \eqref{4.47} and Lemma \ref{le4.6}-Lemma \ref{le4.7}, we have, for $i=1,2$,
\begin{equation}
\label{4.55}
\begin{aligned}
&\tilde{B}_j\sum_{h=1}^2(D_{hi}^2(\phi_{n,j}+\log h)(x_{n,j})b_{j,h})e^{-\frac{\mu_{n,j}}{2}}\\
&=-8\pi\Bigg[\sum_{l\neq j}e^{-\frac{\mu_{n,l}}{2}}D_iG_{n,l}^*(x_{n,j})
+e^{-\frac{\mu_{n,j}}{2}}D_i\sum_{h=1}^2\partial_{y_h}R(y,x)|_{y=y=x_{n,j}}b_{j,h}\tilde{B}_j\Bigg]+o(e^{-\frac{\mu_{n,j}}{2}})\\
&=-8\pi\sum_{l\neq j} e^{-\frac{\mu_{n,l}}{2}}\sum_{h=1}^2D_{x_i}\partial_{y_h}G(y,x)|_{(y,x)=(x_{n,j},x_{n,l})}b_{lh}\tilde{B}_l\\
&\quad-8\pi e^{-\frac{\mu_{n,j}}{2}}\sum_{h=1}^2D_{x_i}\partial_{y_h}R(y,x)|_{x=y=x_{n,j}}b_{j,h}\tilde{B}_j+o(e^{-\frac{\mu_{n,j}}{2}}).
\end{aligned}
\end{equation}
Set $\vec{b}=({\hat{b}}_{1,1}\tilde{B}_1,{\hat{b}}_{1,2}\tilde{B}_1,\cdots,{\hat{b}}_{m,1}\tilde{B}_m,{\hat{b}}_{m,2}\tilde{B}_m)$,
where $\hat{b}_{lh}=\lim_{n\to+\infty}(e^{\frac{\mu_{n,j}-\mu_{n,l}}{2}}b_{lh})$.
Then, by using Theorem 2.A and passing to the limit as $n\to +\infty$, we conclude from \eqref{4.55} that,
$$
D^2f_m({q}_1,{q}_2,\cdots,{q}_m) \cdot \vec{b}= 0.
$$
By the non degeneracy assumption $\mathrm{det}(D_{\Omega}^2f_m(\mathbf{q}))\neq 0$, we  conclude $b_{j,1}=b_{j,2}=0$, $j=1,\cdots, m$.
\end{proof}


\bigskip




\begin{thebibliography}{99}

\bibitem{Aub} T. Aubin, Nonlinear analysis on Manifolds Monge-Amp\'ere equations. Grundlehren der Mathematischen Wissenschaften {\bf 252},
Springer-Verlag, New York, 1982.

\bibitem{Ban} C. Bandle, Isoperimetric inequalities and applications, Pitmann, London, 1980.

\bibitem{bp} S.  Baraket, F.  Pacard, {\em Construction of singular limits for a semilinear elliptic equation in dimension 2},
Calc.  Var.  Partial Differential Equations  \textbf{6}  (1998),  no.  1, 1-38.

\bibitem{B2} D. Bartolucci, {\em Global bifurcation analysis of mean field equations and
the Onsager microcanonical description of two-dimensional turbulence}, arXiv:1609.04139.

%\bibitem{BDeM}
%D. Bartolucci, F. De Marchis, {\em On the Ambjorn-Olesen electroweak condensates},
%{Jour. Math. Phys.} {\bf 53} 073704 (2012); doi: 10.1063/1.4731239.

\bibitem{BdM2} D. Bartolucci, F. De Marchis,
{\em Supercritical Mean Field Equations on convex domains and the Onsager's
statistical description of two-dimensional turbulence}, Arch. Rat. Mech. Anal. {\bf 217}/2 (2015), 525-570;
DOI: 10.1007/s00205-014-0836-8.

\bibitem{BdMM} D. Bartolucci, F. De Marchis, A. Malchiodi, Supercritical conformal metrics on
surfaces with conical singularities, Int. Math. Res. Not. 2011, no. 24, 5625-5643; DOI: 10.1093/imrn/rnq285.

\bibitem{bl} D. Bartolucci, C.S. Lin, {\em Uniqueness Results for Mean Field Equations with Singular Data},
Comm. in P. D. E. {\bf 34} (2009), no. 7-9, 676-702.

\bibitem{BLin3} D. Bartolucci, C.S. Lin, {\em Existence and uniqueness for
Mean Field Equations on multiply connected domains at the critical parameter},
{Math. Ann.} {\bf 359} (2014), 1-44; DOI 10.1007/s00208-013-0990-6.

\bibitem{BJLY}D. Bartolucci, A. Jevnikar, W. Yang, Y. Lee, {\em Uniqueness of bubbling solutions of mean field equations}, arXiv:1704.02354.

\bibitem{bt} D. Bartolucci, G. Tarantello, {\em Liouville type equations with
singular data and their applications to periodic multivortices for the
electroweak theory}, Comm. Math. Phys. {\bf 229} (2002), 3-47.

\bibitem{bt2} D. Bartolucci, G. Tarantello, {\em Asymptotic blow-up analysis for singular Liouville type equations with
applications}, J. D. E. {\bf 262} (2017), 3887-3931.

\bibitem{bav} F. Bavaud, {\sl Equilibrium properties of the Vlasov functional: the generalized Poisson-Boltzmann-Emden
equation}, Rev. Mod. Phys. {\bf 63} (1991), no. 1, 129-149.

\bibitem{beb} J. Bebernes, D. Eberly, Mathematical Problems from Combustion Theory, A. M.
S. 83, Springer-Verlag, New York, 1989.

\bibitem{bm} H. Brezis, F. Merle,
{\em Uniform estimates and blow-up behaviour for solutions of $-\Delta u = V(x)e^{u}$ in two dimensions},
Comm. in P.D.E.  {\bf 16} (1991), no. 8-9, 1223-1253.

\bibitem{bdt} Buffoni, B., Dancer, E.N., Toland, J.F., {\em The sub-harmonic bifurcation of Stokes waves},
Arch. Rat. Mech. Anal. {\bf 152} (2000), no. 3, 24-271.

\bibitem{but} B. Buffoni, J. Toland, {Analytic Theory of Global Bifurcation}, Princeton Univ. Press. 2003.

\bibitem{clmp1} E. Caglioti, P.L. Lions, C. Marchioro, M. Pulvirenti,
{\em A special class of stationary flows for two dimensional Euler equations: a
statistical mechanics description,} Comm. Math. Phys. {\bf 143} (1992), 501-525.

\bibitem{clmp2} E. Caglioti, P.L. Lions, C. Marchioro, M. Pulvirenti,
{\em A special class of stationary flows for two dimensional Euler equations: a
statistical mechanics description. II}, Comm. Math. Phys. {\bf 174} (1995),
229-260.


\bibitem{CCL} S.Y.A. Chang, C.C. Chen, C.S. Lin, {\em Extremal functions for a mean field equation in two dimension},
in: Lecture on Partial Differential Equations, New Stud. Adv. Math. {\bf 2} Int. Press, Somerville, MA, 2003, 61-93.

\bibitem{cli1} W. X. Chen, C. Li, {\em Classification of solutions of some nonlinear elliptic equations,}
Duke Math. J.  {\bf 63} (1991), no. 3, 615-622.

\bibitem{CLin1} C.  C.  Chen,    C.  S.  Lin, {\em Sharp estimates for solutions of multi-bubbles in compact Riemann surface.}
{ Comm.  Pure Appl.  Math. } \textbf{55} (2002), 728-771.

\bibitem{CLin2} C. C. Chen, C.S. Lin, {\em Topological Degree for a mean field equation on Riemann surface},
                Comm. Pure Appl. Math. {\bf 56} (2003), 1667-1727.

\bibitem{CLin5} C. C. Chen, C. S. Lin, {\em Mean Field Equation of Liouville Type
with Singular Data: Topological Degree},  Comm. Pure Appl. Math.  \textbf{68} (2015), no. 6, 887-947.

\bibitem{clw} C.  C.  Chen,    C.  S.  Lin, G. Wang, {\em Concentration phenomena of two-vortex solutions
in a Chern-Simons model.}  {Ann. Sc. Norm. Super. Pisa Cl. Sci.}  \textbf{3} (2004), no. 2, 367-397.

\bibitem{dem2} F. De Marchis, {\em Generic multiplicity for a scalar field equation on compact surfaces},
J. Funct. An. \textbf{259} (2010), no. 8, 2165-2192.

\bibitem{EGP} P. Esposito, M. Grossi, A. Pistoia, {\em On the existence of blowing-up solutions
for a mean field equation}, Ann. Inst. H. Poincar\'e Anal. Non Lin\'eaire {\bf 22} (2005), no. 2, 227-257.

\bibitem{ESp} G.L. Eyink, H. Spohn, {\em Negative temperature states and large-scale,
long-lived vortices in two dimensional turbulence},
{J. Stat. Phys.} {\bf 70} (1993), no. 3-4, 87-135.


\bibitem{ESr} G.L. Eyink, K.R. Sreenivasan, {\em Onsager and the theory of hydrodynamic turbulence},
{Rev. Mod. Phys.} {\bf 78} (2006), 833-886.

\bibitem{Fang} H. Fang, M. Lai, {\em On curvature pinching of conic 2-spheres},
Calc. Var. P.D.E. (2016), 55:118.

\bibitem{Gel} I. M. Gelfand, Some problems in the theory of quasi-linear equations,
Amer. Math. Soc. Transl. 29 (1963) no. 2, 295-381.

\bibitem{GlGr} F. Gladiali., M. Grossi, {\em Some Results for the Gelfand's Problem}, Comm. P.D.E. {\bf 29} (2004), no. 9-10,
1335-1364.

\bibitem{GOS} M. Grossi, H. Ohtsuka, T. Suzuki, {\em Asymptotic non-degeneracy of the multiple blow-up solutions of the Gel'fand problem in two space dimensions}, 
Adv. Diff. Eq. {\bf 16} (2011), no. 1-2, 145-164.

\bibitem{GM1} C. Gui, A. Moradifam, {\em The Sphere Covering Inequality and Its Applications}, Invent. Math., to appear.

\bibitem{Gu} B. Gustafsson, {\em On the convexity of a solution of Liouville's
equation equation}, Duke Math. Jour.  {\bf 60} (1990), no. 2, 303-311.


\bibitem{K} M.K.H. Kiessling,
{\em Statistical mechanics of classical particles with logarithmic interaction},
Comm. Pure Appl. Math. {\bf 46} (1993), 27-56.

\bibitem{KL} M.K.H. Kiessling, J. L. Lebowitz {\em The Micro-Canonical Point Vortex Ensemble:
Beyond Equivalence}, Lett. Math. Phys. {\bf 42} (1997), 43--56.

\bibitem{KMdP} M. Kowalczyk, M. Musso, M. del Pino,  {\em Singular limits in
Liouville-type equations}, Calc. Var. P.D.E. {\bf 24} (2005), no. 1, 47-81.

\bibitem{yy} Y.Y. Li,  {\em Harnack type inequality: the method of moving planes},
Comm. Math. Phys.  {\bf 200} (1999), 421--444.

\bibitem{ly2} C.  S.  Lin, S.     Yan, {\em On the Chern-Simons-Higgs equation: Part II, local
uniqueness and exact number of solutions}, {preprint.}

\bibitem{linwang} C.S. Lin, C.L. Wang, {\em Elliptic functions, Green functions
and the mean field equations on tori}, Ann. of Math. {\bf 172} (2010), no. 2, 911-954.

\bibitem{Mal1} A. Malchiodi, {\em Topological methods for an elliptic equation with exponential nonlinearities},
Discr. Cont. Dyn. Syst. {\bf 21} (2008), 277-294.

\bibitem{Mal2} A. Malchiodi, {\em Morse theory and a scalar field equation on compact
surfaces}, Adv. Diff. Eq. {\bf 13} (2008), 1109-1129.

\bibitem{New} P.K. Newton, {\em The N-Vortex Problem: Analytical Techniques}, Appl. Math. Sci. {\bf 145},
Springer-Verlag, New York, 2001.

\bibitem{On}  L. Onsager, {\em Statistical hydrodynamics}, {Nuovo Cimento} {\bf 6} (1949), no. 2, 279-287.

\bibitem{suz} T. Suzuki, {\em Global analysis for a two-dimensional elliptic eiqenvalue problem with the exponential
                nonlinearly}, Ann. Inst. H. Poincar\'e Anal. Non Lin\'eaire {\bf 9} (1992), no. 4, 367-398.

\bibitem{suzC} T. Suzuki, Free Energy and Self-Interacting Particles, PNLDE
{\bf 62}, Birkhauser, Boston, (2005).

\bibitem{suzD} T. Suzuki, {\em Some remarks about singular perturbed solutions for 
Emden-Fowler equation with exponential nonlinearity}. In: Functional Analysis and
Related Topics. 1991, Kyoto. Lecture Notes in Math., {\bf 1540}. Berlin: Springer (1993).

\bibitem{T} G. Tarantello,
{\em Multiple condensate solutions for the Chern-Simons-Higgs theory},
{J. Math. Phys.} {\bf 37} (1996), 3769-3796.

\bibitem{tar} G. Tarantello, {Self-Dual Gauge Field Vortices: An Analytical Approach},
PNLDE {\bf 72}, Birkh\"auser Boston, Inc., Boston, MA, 2007.

\bibitem{Troy} M. Troyanov, {\em Prescribing curvature on compact surfaces with
conical singularities}, Trans. Amer. Math. Soc. {\bf 324} (1991), 793-821.

\bibitem{yang} Y. Yang, Solitons in Field Theory and Nonlinear Analysis,
Springer Monographs in Mathematics {\bf 146}, Springer, New York, 2001.

\end{thebibliography}
\end{document}